\documentclass{article}

\usepackage[title]{appendix}
\usepackage{amsthm,amsmath,amssymb,amsfonts}
\usepackage{bbold}
\usepackage{xcolor}
\usepackage{algorithm}
\usepackage{algorithmicx}
\usepackage{algpseudocode}
\usepackage{ulem}

\usepackage{geometry}
\usepackage{comment}
\usepackage{caption}
\usepackage{array}
\usepackage{makecell}
\usepackage{longtable}
\usepackage{tabularx}
\usepackage{capt-of}
\usepackage{multirow}
\usepackage[justification=justified]{caption}
\usepackage{lmodern}
\usepackage{listings}
\usepackage{calrsfs}
\newcommand{\newpar}{\par\medskip}
\usepackage{dsfont}
\allowdisplaybreaks
\usepackage{here}
\usepackage{graphicx}
\usepackage{subcaption}
\RequirePackage[numbers]{natbib}
\usepackage[bookmarks,colorlinks=true,linkcolor=blue,urlcolor=red]{hyperref}
\RequirePackage{hypernat}

\usepackage{easyeqn}

\usepackage{authblk}

\def\bA{\boldsymbol{A}}

\def\bD{\boldsymbol{D}}

\def\bI{\boldsymbol{I}}

\def\bL{\boldsymbol{L}}
\def\bM{\boldsymbol{M}}

\def\bP{\boldsymbol{\mathrm{P}}}

\def\bS{\boldsymbol{S}}
\def\bT{\boldsymbol{T}}
\def\probLeftEigenvectors{\boldsymbol{U}}
\def\probRightEigenvectors{\boldsymbol{V}}
\def\bW{\boldsymbol{W}}
\def\bX{\boldsymbol{X}}
\def\bY{\boldsymbol{Y}}
\def\bZ{\boldsymbol{Z}}
\def\bPi{\boldsymbol{\Pi}}

\def\ba{\mathbf{a}}

\def\be{\mathbf{e}}

\def\bs{\mathbf{s}}

\def\probLeftEigenvectors{\mathbf{u}}
\def\probRightEigenvectors{\mathbf{v}}

\newcommand{\E}{\operatorname{\mathbb{E}}}
\renewcommand{\P}{\operatorname{\mathbb{P}}}

\def\T{\mathrm{T}}
\def\t{\times}

\def\nsize{n}
\def\RR{\mathbb{R}}

\def\PP{\mathbb{P}}

\def\ev{\mathbf{e}}

\def\gv{\mathbf{g}}
\def\hv{\mathbf{h}}

\def\ntopics{K}

\def\documentTopicMatrix{\bW}

\def\probLeftEigenvectors{\mathbf{U}}
\def\probRightEigenvectors{\mathbf{V}}
\def\probEigenvalues{\mathbf{L}}
\def\empiricalLeftEigenvectors{\hat{\probLeftEigenvectors}}
\def\empiricalRightEigenvectors{\hat{\probRightEigenvectors}}
\def\adjacencyEigenvalues{\hat{\probEigenvalues}}
\def\noiseMatrix{\mathbf{N}}

\def\documentTopicMatrixEstimate{\hat{\documentTopicMatrix}}

\def\orthMatrix{\mathbf{O}}
\def\errorAdjacency{\beta(\bX, \bPi)}
\def\condNumber{\kappa}

\def\targetMatrix{\mathbf{G}}
\def\basisMatrix{\mathbf{H}}

\def\noisyTargetMatrix{\tilde{\targetMatrix}}
\def\basisMatrixEstimate{\hat{\mathbf{H}}}

\newtheorem{theorem}{Theorem}
\newtheorem{lemma}{Lemma}
\newtheorem{corollary}{Corollary}
\newtheorem{proposition}{Proposition}
\newtheorem{assumption}{Assumption}

\newtheorem{remark}{Remark}

\def\argmax{\mathop{\rm arg\,max}}

\def\det{\mathop{\rm det}}
\def\diag{\mathop{\rm diag}\nolimits}

\algrenewcommand\algorithmicrequire{\textbf{Input:}}
\algrenewcommand\algorithmicensure{\textbf{Output:}}

\newcommand\numberthis{\addtocounter{equation}{1}\tag{\theequation}}

\title{Assigning Topics to Documents by Successive Projections}

\providecommand{\keywords}[1]
{
  \small	
  \textbf{\textit{Keywords---}} #1
}
\author[1,3]{Olga Klopp}
\author[2]{Maxim Panov}
\author[3]{Suzanne Sigalla}
\author[3]{ Alexandre B. Tsybakov}
\affil[1]{ESSEC Business School, France}
\affil[2]{Skoltech, Russia}
\affil[3]{CREST, ENSAE, France}

\begin{document}

\maketitle

\begin{abstract}
    Topic models provide a useful tool to organize and understand the
    structure of large corpora of text documents, in particular, to discover hidden
    thematic structure.  Clustering documents from big unstructured corpora into topics 
    is an important task in various areas, such as image analysis, e-commerce, social networks, population
    genetics.  A common approach to topic modeling is to associate each topic with a probability distribution
    on the dictionary of words and to consider each document as a mixture of topics.
    Since the number of topics is typically substantially smaller than the size of the corpus
    and of the dictionary, the methods of topic modeling can lead to a dramatic dimension reduction.
    In this paper, we study the problem of estimating topics distribution for each
    document in the given corpus, that is, we focus on the clustering aspect of the problem.
    We introduce an algorithm that we call Successive  Projection  Overlapping  Clustering (SPOC) inspired by the Successive Projection Algorithm
    for separable matrix factorization.  
    This algorithm is simple to implement and  computationally fast. We establish theoretical guarantees 
    on the performance of the SPOC algorithm, in particular, near matching minimax upper and lower bounds on its estimation risk.
    We also propose a new method that estimates the number of topics. We complement our theoretical results with a numerical study on synthetic and semi-synthetic data to analyze the performance of this new algorithm in practice.
    One of the conclusions is that the error of the algorithm grows at most logarithmically
    with the size of the dictionary, in contrast to what one observes for Latent Dirichlet Allocation. 
\end{abstract}

\section{Introduction}
\label{sec:intro}
  Assigning  topics to documents is an important task in several applications.
  For example, press agencies need to identify articles of interest to readers based on the topics of articles that they have read in the past.
  Analogous goals are pursued by many other text-mining applications such as, for example, recommending blogs from among
  the millions of blogs available.
  A popular approach to the problem of estimating hidden thematic structures in a corpus of documents is based on topic modeling. 
  Topic models have
  attracted a great deal of attention in the past two decades. Beyond text mining, they were used in areas, such as population genetics~\cite{bicego2012investigating,pritchard2000inference}, social networks~\cite{mccallum2005author,curiskis2020evaluation}, image analysis~\cite{li2010building,zhu2017scene}, e-commerce~\cite{palese2018relative,yuan2018topic}.


  In this paper, we adopt the \textit{probabilistic Latent Semantic Indexing} (pLSI) model introduced in~\cite{hofmann1999probabilistic}. The pLSI model deals with three types of variables, namely, documents, topics and words. Topics are latent variables, while the observed variables  are words and documents.  Assume that we have a dictionary of $p$ words and a collection of $n$ documents.  Documents are sequences of words from the dictionary. The number of topics is denoted by~$K$. Usually, $K\ll \min(p,n)$. Throughout this paper, we assume that $2\le K\le \min(p,n)$. The pLSI model assumes that the probability of occurrence of word $j$ in a document discussing topic $k$ is independent of the document. Therefore, by the total probability formula, 
  \begin{align*}
    \mathbb{P}(\text {word } j | \text {document } i) = \sum_{k = 1}^{K} \mathbb{P}(\text {topic } k | \text {document } i) \mathbb{P}(\text {word } j | \text {topic } k).
  \end{align*}
  Introducing the notation $\Pi_{ij}:=\mathbb{P}(\text {word } j | \text {document } i)$, $W_{ik} := \PP(\text{topic }k\rvert\text{document }i)$ and $A_{kj} := \PP(\text{word }j\rvert\text{topic } k)$ we may write $\Pi_{ij} = W_{i}^\T A_{j}$, where $W_{i}=(W_{i1}, \dots, W_{iK})^\T \in \left[0,1\right]^K$ is the topic probability vector for  document $i$ and $A_{j}=(A_{1j}, \dots, A_{Kj})^\T\in \left[0,1\right]^K$ is the vector of word $j$ probabilities under topics $k=1,\dots, K$. In matrix form, 
  \begin{align}\label{factoriz}
    \bPi = \bW \bA,
  \end{align}
  where $\bPi$ is the document-word matrix of size $n\times p$ with entries $\Pi_{ij}$, $\bW:= \left(W_{1}, \dots, W_{n} \right)^\T$ is the document-topic matrix of size $n\times K$ and $\bA := \left(A_{1}, \dots, A_{p} \right)$ 
  is the topic-word matrix of size $K\times p$. The rows of these matrices are probability vectors, 
  \begin{align}\label{factoriz1}
  \sum_{m = 1}^K W_{im}=1,~\sum_{j = 1}^p A_{kj}=1,~\sum_{j=1}^p \Pi_{ij}=1~\text{for any }i= 1,\dots,n,~k= 1,\dots K.
  \end{align} 
  Unless otherwise stated, we will assume throughout the paper that $\bPi, \bW, \bA$ are matrices with non-negative entries satisfying~\eqref{factoriz1}. 

  The value $\Pi_{ij}$ is the probability of occurrence of word $j$ in document $i$. It is not available but we have access to the corresponding empirical frequency $X_{ij}$. 
  Thus, we have a document-word matrix $\bX=(X_{ij})$ of size $n \times p$ such that for each document $i$ in $1,\dots, n,$ and each word $j$ in $1,\dots, p$, the entry $X_{ij}$ is the observed frequency of word $j$ in document $i$.  Let $N_i$   denote the (non-random) number of sampled words in document $i$.
  In what follows, we assume that, for each document-word vector $X_i = (X_{i1}, \dots,X_{ip})^\T$,  the corresponding vector of cumulative counts $N_i X_i$ follows a $\text{Multinomial}_p(N_i,\Pi_i)$ distribution,
  where $\Pi_i := \E(X_i) = (\Pi_{i1}, \dots,\Pi_{ip})^\T$. We also assume that $X_1,\dots,X_n$ are independent random vectors. We will denote by $\mathbb{P}_{\bPi}$
  the probability measure corresponding to the distribution of $\bX$.
  
  We can write the observation model in a ``signal + noise" form:
  \begin{align}
  \label{eq:anchor_docs_model}
    \bX = \bPi + \bZ = \bW \bA + \bZ,
  \end{align}
  where $\bZ := \bX - \bPi$ is a zero mean noise. 
  The objective in topic modeling is to estimate  matrices $\bA$ and $\bW$ based on the observed frequency matrix $\bX$ and on the known $N_1, \dots, N_n$. The recovery of $\bA$ and the recovery of $\bW$ address different purposes. An estimator of matrix $\bA$ identifies the topic distribution on the dictionary. An estimator of $\bW$ indicates the topics associated to each document.
  
  Estimation of $\bW$ has multiple applications and has been extensively discussed in the literature, mainly in the Bayesian perspective. The  focus was on Latent Dirichlet Allocation (LDA) and related techniques (see Section~\ref{sec:comparison} for more details and references).
  These methods are computationally slow and, to the best of our knowledge, no theoretical guarantees on their performance are available.

  On the other hand, estimation of matrix $\bA$ is well-studied in the theory. Several papers provide bounds on the performance of different estimators of $\bA$. We give a more detailed account on this work in Section~\ref{sec:comparison}.  Most of the results use the \textit{anchor word assumption} postulating that for every topic there is at least one word, which occurs only in this topic, see~\cite{arora2013practical,bing2018fast,Ke2017}. At first sight, it seems that results on estimation of matrix $\bA$ can be applied to estimation of $\bW$ by simply taking the transpose of~\eqref{factoriz1} and inverting the roles of these two matrices. However, such an argument is not valid since the resulting models are different.  Indeed, the rows of matrix $\bX^\T$ are not independent and the rows of matrices $\bPi^\T, \bA^\T, \bW^\T$ do not sum up to 1, which leads to a different statistical analysis. 
  
In the present paper, we change the framework compared to ~\cite{arora2013practical,bing2018fast,Ke2017} by focusing on estimation of matrix $\bW$ rather than $\bA$. We introduce the following assumption. 
   \begin{assumption}[Anchor document assumption] 
  \label{ass:anchorDoc}
    For each topic \(k = 1, \dots, \ntopics,\) there exists at least one document \(i\) (called an anchor document) such that \(W_{ik} = 1\) and \(W_{il} = 0\) for all \(l \not= k\).
  \end{assumption}
  {Both anchor word and anchor document assumptions are very relevant in real word applications.}
Since each document is identified with a mixture of $K$ topics, anchor document assumption means that, for each topic, there is a document devoted solely to this topic. To illustrate the anchor document assumption, consider
the Associated Press data set~\cite{Harman1993}, which is a collection of 2246 articles published by this press agency mostly around 1988. An application of the pLSI model fitted via the LDA method with $K = 2$ leads to two well-shaped topics ``finance'' and ``politics''. We refer to~\cite[Section~6]{Silge2020} 
for the details of the analysis. The first 10 rows of the estimator of matrix \(\bW\) are presented in Table~\ref{tab:associated_matrix}.  Notice that documents 6 and 8 in Table~\ref{tab:associated_matrix} can be considered as anchor documents. For example, document 6 has the weight of the second topic estimated as \(0.999\). A closer look at the most frequent words in this document (Noriega, Panama, Jackson, Powell, administration, economic, general) tells us that, indeed, this article corresponds solely to the topic ``politics'' -- it is about the relationship between the American government and the Panamanian leader Manuel Noriega. 

\begin{table}[t]
  \centering
  \captionof{table}{The first ten rows of the estimated matrix \(\bW\) for the Associated Press data set.}  
  \begin{tabular}{|c|c|c|}
    \hline
    Document & Finance & Politics \\
    \hline
    \hline
    1  & 0.248 & 0.752 \\   
    2  & 0.362 & 0.638 \\  
    3  & 0.527 & 0.473 \\  
    4  & 0.357 & 0.643 \\ 
    5  & 0.181 & 0.819 \\
    6  & 0.001 & 0.999 \\
    7  & 0.773 & 0.227 \\ 
    8  & 0.004 & 0.996 \\
    9  & 0.967 & 0.033 \\
    10 & 0.147 & 0.953 \\
    \hline
  \end{tabular}
  \label{tab:associated_matrix}
\end{table}

  
  Our approach to estimation of matrix $\bW$ that we call Successive  Projection  Overlapping  Clustering (SPOC) is inspired by the Successive Projection Algorithm (SPA) initially proposed for non-negative matrix factorization~\cite{Araujo2001} and further used by~\cite{Gillis2014,panov2017consistent,Mao2020} in the context of mixed membership stochastic block models. 
  The idea of such methods is to start with the singular value decomposition (SVD) of matrix $\bX$, and launch an iterative procedure that, at each step,  chooses the maximum norm row of the matrix composed of singular vectors  
  and then projects on the linear subspace orthogonal to the selected row. From a geometric perspective, the rows of the matrix composed of singular vectors of $\bPi$ belong to a simplex in $\RR^K$. The documents can be identified with some points in this simplex and the anchor documents with its vertices. Our algorithm iteratively finds estimators of the vertices, based on which we finally estimate $\bW$.
  
  Note that the idea of exploiting simplex structures for estimation of matrix $\bA$ rather than $\bW$ was previously developed in, for example, ~\cite{arora2013practical,Ding2013,Ke2017}, among others. For example, the method to estimate $\bA$ suggested in~\cite{Ke2017} is based on an exhaustive search over all size $K$ subsets of $\{1,\dots, p\}$.   Its goal is  to select $K$ vertices of a $p$-dimensional simplex and its computational cost is at least of the order $p^K$. Our algorithm for estimating $\bW$ recovers the vertices of much less complex object, which is a $K$-dimensional simplex (recall that $K\ll p$), and its computational cost is of the order $\max(p, n) K + nK^2$. Here, $\max(p, n) K$ and $n K^2$ are the costs of performing a truncated singular value decomposition and the SPA algorithm, respectively. Another important point is that existing simplex-based methods for estimation of matrix $\bA$ require  the number $K$ of topics to be known. In the present paper, we propose a procedure that is adaptive to 
  unknown $K$. 
  
  {Our theoretical results deal only with the problem of estimating the topic-document matrix $\bW$, for which the theory was not developed in prior work. But in practice, our method can be used for estimation of matrix $\bA$ as well. Based on the SPOC estimator of $\bW$, we immediately obtain an estimator of matrix $\bA$ by a computationally fast procedure (see Section \ref{sec:anchor}). Our simulation studies (see Section~\ref{sec:topic_word_experiments} of the Appendix) indicate that this estimator exhibits a behavior similar to LDA on average while being more stable.}
  
  
  
  {In this paper, we prove that the SPOC estimator of $\bW$ converges in the Frobenius norm and in the $\ell_1$-norm with the rates $\sqrt{n/N}$ and $n/\sqrt{N}$ (up to a weak factor\footnote{In what follows, we mean by {\it weak factor} a small power of $K$ multiplied by a term logarithmic in the parameters of the problem. We will ignore weak factors when discussing the convergence rates.}), respectively, assuming that  $N_i=N$ for $i=1,\dots,n$. We also prove lower bounds of the order $\sqrt{n/N}$ and $n/\sqrt{N}$, respectively, implying near optimality of the proposed method. 
  } One of the conclusions, both from the theory and the numerical experiments, is that the error of the SPOC algorithm does not grow significantly with the size of the dictionary $p$, in contrast to what one observes for Latent Dirichlet Allocation.  We also introduce an estimator for the number $K$ of topics, which is usually unknown in practice. We show that SPOC algorithm using the estimator of $K$ preserves its optimal properties in this more challenging setting.
  
  
  

  The rest of the paper is organized as follows. 
  In Section~\ref{sec:anchor},  we introduce the SPOC algorithm.  Section~\ref{sec:consistency} contains the main results on the convergence rate of the algorithm and the minimax lower bound for estimation of \(\bW\).  Section~\ref{sec:comparison} is devoted to discussion of the prior work.  In Section~\ref{sec:simulations}, we present numerical experiments for synthetic and real-world data in order to illustrate our theoretical findings. Finally, in Section~\ref{sec:discussion} we summarize the outcomes of the study. The proofs are given in the Appendix.

\section{Notation}
\label{sec:notation}
  For any matrix $\bM=(M_{ij})\in \mathbb{R}^{n \times k}$, we denote by $\|\bM\|$ its spectral norm, i.e., its maximal singular value, by $\|\bM\|_F$ its Frobenius norm, and by $\|\bM\|_1=\sum_{i = n}^k\sum_{j = 1}^k\vert M_{ij}\vert$ its $\ell_1$-norm. We also consider the maximum $\ell_1$-norm of its rows $\|\bM\|_{1,\infty} =\underset{1 \leq i \leq n}{\max} \sum_{j = 1}^k\vert M_{ij}\vert$. We denote by $\lambda_j(\bM)$ the $j$th singular value and by $\lambda_{\min}(\bM)$ the smallest  singular value of $\bM$. Assuming that matrix $\bM$ has rank $K$ we consider its singular values $\lambda_1(\bM) \geq \lambda_2(\bM) \geq \dots \geq \lambda_K(\bM)>0$ and its condition number $\kappa(\bM) = {\lambda_1(\bM)}/{\lambda_K(\bM)}$.  If  $J$ is a non-empty subset of rows of matrix $\bM$ the notation $\bM_J$ is used for a matrix in $\mathbb{R}^{|J| \times k}$ obtained from $\bM$ by keeping only the rows in $J$. We denote by $\bI_K$ the $K\times K$ identity matrix, and by $(\be_1,\dots,\be_n)$ the canonical basis of $\RR^n$.  For any vector $u\in \RR^d$, we denote by $\|u\|_2$ its Euclidean norm. Throughout the paper, we use the notation $\orthMatrix$
  for orthogonal matrices and $\bP$ for permutation matrices. We denote by $\mathcal{P}$ the set of all permutation matrices in $\RR^{\ntopics \t \ntopics}$. We denote by $c, C$ positive constants than may vary from line to line 

\section{Successive Projection Overlapping Clustering}
\label{sec:anchor}
  In this section, we introduce the \textit{Successive Projection Overlapping Clustering (SPOC)} algorithm for estimation of matrix \(\bW\). It is an analog, in the context of topic models, of the algorithm  proposed in~\cite{panov2017consistent} for the problem of parameter estimation in Mixed Membership Stochastic Block Model.

  In order to explain the main idea of the algorithm, we start with considering the singular value decomposition (SVD) of matrix \(\bPi\):
  \begin{EQA}[c]
  \label{svc:Pi}
    \bPi = \probLeftEigenvectors \probEigenvalues \probRightEigenvectors^{\T},
  \end{EQA}
  where \(\probLeftEigenvectors = \left[U_{1}, \dots, U_{\ntopics}\right] \in \RR^{n \t \ntopics}\) and \(\probRightEigenvectors = \left[V_{1}, \dots, V_{\ntopics}\right] \in \RR^{p \t \ntopics}\) are the matrices of left and right singular vectors and \(\probEigenvalues \in \RR^{\ntopics \t \ntopics}\) is a diagonal matrix of the corresponding singular values. We have used here the fact that  the rank of matrix \(\bPi\) is at most~$K$. Recall that we assume $K\le \min(p,n)$. A key observation is that, if $\lambda_K(\bPi)>0$
  and Assumption~\ref{ass:anchorDoc} is satisfied, then matrix \(\probLeftEigenvectors\) can be represented as 
  \begin{equation}
  \label{UWH}
    \probLeftEigenvectors = \bW \bf H,
  \end{equation}
  where \({\bf H}\in  \RR^{\ntopics \t \ntopics}\) is a full rank matrix (cf. Lemma~\ref{lemma:matrix_eigen0} in the Appendix).
  Thus, the rows of matrix \(\probLeftEigenvectors\) belong to a simplex in $\RR^K$ with vertices given by the rows of matrix \(\bf H\).

  The empirical counterparts of $\bf{U,L,V}$ are obtained from the SVD of the observed matrix \(\bX\):
  \begin{EQA}[c]\label{svc:X}
    \bf \bX = \empiricalLeftEigenvectors \adjacencyEigenvalues \widehat{\probRightEigenvectors}^{\T} + \empiricalLeftEigenvectors_1 \adjacencyEigenvalues_1 \widehat{\probRightEigenvectors}_1^{\T},
  \end{EQA}
  where $\empiricalLeftEigenvectors = [\widehat{U}_{1}, \dots, \widehat{U}_{\ntopics}]$ and $\widehat{\probRightEigenvectors} = [\widehat{V}_{1}, \dots, \widehat{V}_{\ntopics}]$ are, respectively, the matrices of left and right singular vectors of $\bX$ corresponding to its $K$ leading singular values $\widehat{\lambda}_{1} \geq \dots \geq \widehat{\lambda}_{\ntopics}$; $\widehat{\probEigenvalues} = \diag\{\widehat{\lambda}_{1}, \dots, \widehat{\lambda}_{\ntopics}\}$, and $\empiricalLeftEigenvectors_1 \adjacencyEigenvalues_1 \widehat{\probRightEigenvectors}_1^{\T}$ is the singular value decomposition of $\bX - \empiricalLeftEigenvectors \adjacencyEigenvalues \widehat{\probRightEigenvectors}^{\T}$.
  Due to matrix perturbation results (see Appendix~\ref{subsec:perturbation}), there exists an orthogonal matrix ${\bf O}$ such that $\empiricalLeftEigenvectors$ is a good approximation for $\probLeftEigenvectors {\bf O}$.
  We can write
  \begin{EQA}[c]\label{eq:umodel}
    \empiricalLeftEigenvectors = \probLeftEigenvectors {\bf O} + {\bf N} = \bW \bf H {\bf O} + {\bf N},
  \end{EQA}
  where ${\bf N}$ is a "small enough" noise matrix. 
Having obtained $\empiricalLeftEigenvectors$ from the SVD of $\bX$, we then apply the \textit{Successive Projection Algorithm} (SPA)~\cite{Araujo2001,Gillis2014} to estimate the matrix \(\bf H {\bf O}\) in model~\eqref{eq:umodel}. 

  \begin{algorithm}
    \caption{SPA}
    \begin{algorithmic}[1]
      \Require{Matrix \(\bM\in \mathbb{R}^{n\times K}\) and integer \(r\le n\).}
      \Ensure{Set of indices $J\subseteq \{1,\dots,n\}$.}
      \State Initialize: $\bS_0=\bM^{\T}$, $J_0=\emptyset$.
      \State For $t=1,\dots,r$ do:\newline
      -- Find $i(t) = \argmax_{i=1,\dots,n}\|\bs_i\|_2$, where $\bs_i$'s are the column vectors  of $\bS_{t-1}$.\newline
      -- Set 
      $
      \bS_{t}= \left(\bI_K - \frac{\bs_{i(t)}\bs_{i(t)}^\T}{\|\bs_{i(t)}\|_2^2}\right)\bS_{t-1} 
      $,
      $\quad J_t=J_{t-1}\cup \{i(t)\}$.
      \State Set $J=J_r$.
    \end{algorithmic}
  \label{algo:SPA}
  \end{algorithm}
  Applied to matrix $\bM =\empiricalLeftEigenvectors$ and $r=K$ 
  this algorithm  finds the rows of matrix \(\empiricalLeftEigenvectors\) with the maximum Euclidean norm and then projects on the subspace orthogonal to these rows and repeats the procedure until $K$ rows are selected. The main idea underlying the SPA is that the maximum of the Euclidean norm of a vector on a simplex is attained at one of its vertices.

  In the noiseless case when ${\bf N}=0$, it can be shown  that  if Assumption~\ref{ass:anchorDoc} holds then 
   $\hat{\bf U}_J={\bf H} {\bf O}$, where $J$ is the set of $K$ rows of \(\hat{\bf U}\) selected after $K$ steps of SPA. 
  In the noisy case, we need additional assumptions on  the noise level to ensure that SPA extracts documents close to anchor ones, which leads to an accurate enough estimator $\hat{\bf H}$ of $\bf H {\bf O}$ (see Appendix~\ref{sec:consistencySPA} for the precise statement). Once we have such an  \(\hat{\bf H}\), the final step is to define the following estimator of matrix \(\bW\):
  \begin{EQA}[c]
    \documentTopicMatrixEstimate = \empiricalLeftEigenvectors \hat{\bf H}^{-1}.
  \end{EQA}
  This definition is valid only if matrix $\hat{\bf H}$ is non-degenerate, which is true with high probability under suitable assumptions (cf. Section~\ref{sec:consistency}).
  
  An additional potentially useful step is to apply preconditioning to matrix \(\empiricalLeftEigenvectors\), which leads to improved bounds on the performance of the algorithm in the presence of noise, see~\cite{Gillis2015,Mizutani2016}. Preconditioned SPA  is defined as follows. Let $r=K$ and let $\ba_1,\dots,\ba_n$ be the column vectors of matrix $\bM^\T$. Let $\bL^*\in \mathbb{R}^{K\times K}$ be the solution of the minimization problem 
  \begin{equation}
    \min_{\bL\succ 0:\ \ \max_i \ba_i^\T\bL\ba_i\le 1} - \log \det \bL .
  \end{equation}
  Matrix $\bL^*$ defines the minimum volume ellipsoid centered at the origin that contains $\ba_1,\dots,\ba_n$.
  The preconditioned SPA is defined by Algorithm~\ref{algo:SPA} initialized with $\bS_0=(\bL^*)^{1/2}\bM^\T$
  rather than with $\bS_0=\bM^\T$. 
  
  The SPOC algorithm for topic modeling is summarized in Algorithm~\ref{algo:SPOC}.

  \begin{algorithm}
    \caption{SPOC (respectively, preconditioned SPOC)}
    \begin{algorithmic}[1]
      \Require{Observed matrix \(\bX\) and number of topics \(\ntopics\).}
      \Ensure{Estimated document-topic matrix \(\hat{\bW}\).}
      \State Get the rank \(\ntopics\) SVD of  \(\bX\colon \empiricalLeftEigenvectors \adjacencyEigenvalues \empiricalRightEigenvectors^{\T}\).
      \State Run SPA (respectively, preconditioned SPA) with input \((\empiricalLeftEigenvectors, \ntopics)\), which outputs a set of indices \(J\) with cardinality \(\ntopics\).
      \State Set \(\hat{\bf H}: = \empiricalLeftEigenvectors_J\).
      \State Set \(\documentTopicMatrixEstimate: = \empiricalLeftEigenvectors \hat{\bf H}^{-1}\).
    \end{algorithmic}
  \label{algo:SPOC}
  \end{algorithm}
{Based on the SPOC estimator $\hat{\bW}$ of matrix $\bW$, it is possible to construct an estimator for matrix $\bA$ in a straightforward way. Indeed, given the decompositions \eqref{svc:Pi} and \eqref{UWH}, we can use the definition $\bPi = \bW \bA$ and deduce that $\bA = \bf{H} \probEigenvalues \probRightEigenvectors^{\T}$. A  direct sample-based estimator of~$\bA$ is then given by 
  \begin{equation}
  \label{A_estimate}
    \hat {\bA}= \hat{\bf{H}} \adjacencyEigenvalues \widehat{\probRightEigenvectors}^{\T}.
  \end{equation}
  In order to illustrate the usefulness of this estimator, we have performed its experimental comparison with LDA, see  Section~\ref{sec:topic_word_experiments} of the Appendix.}

\section{Main results}
\label{sec:consistency}
In this section, we provide bounds on the performance of SPOC algorithm. We first prove deterministic bounds assuming that $\bX$ is some fixed matrix close enough to $\bPi$ in the spectral norm. Next, we combine these results with a concentration inequality for $\|\bX - \bPi\|$ when $\bX$ is distributed according to $\mathbb{P}_{\bPi}$ in order to obtain a bound on the estimation error with high probability under our statistical model.

\subsection{Deterministic bounds}
  A key step in analyzing the performance of SPOC algorithm is to show that $\hat{\bf U}$ is close to an orthogonal transformation $\probLeftEigenvectors \orthMatrix$ of the population matrix $\probLeftEigenvectors$. 
  The next lemma gives a bound on the maximal \(\ell_2\)-distance between the rows of \(\empiricalLeftEigenvectors\) and \(\probLeftEigenvectors \orthMatrix\) for some orthogonal matrix \(\orthMatrix\),
  which will allow us to deduce an upper bound on the error  of SPA from the results of~\cite{Gillis2015,Mizutani2016} (see Appendix~\ref{sec:consistencySPA} for the details). 
  We state this lemma as a deterministic result where $\bX$ is some fixed matrix close enough to $\bPi$ in the spectral norm.
  Recall that  \(\lambda_{1}(\documentTopicMatrix)\) is the maximum singular value of matrix \(\documentTopicMatrix\), $\lambda_{\ntopics}( \bPi)$ is the $K$th singular value of matrix $\bPi$, and \(\condNumber(\documentTopicMatrix)\), \(\condNumber(\bPi)\) are the condition numbers of matrices \(\documentTopicMatrix\) and $\bPi$, respectively. 
  Assuming that $\lambda_{\ntopics}(\bPi)>0$ (that is, $\bPi$ is a rank $K$ matrix) we define the values 
  \begin{EQA}[c]
    \beta_{i}(\bX, \bPi)
    =
    \ntopics^{1/2} \condNumber^2(\bPi) ~ \frac{\|\ev_{i}^{\T} \bX\|_{2} \|\bX - \bPi\|}{\lambda_{\ntopics}^{2}(\bPi)}
    +
    \frac{\|\ev_{i}^{\T}(\bX - \bPi) \|_{2}}{\lambda_{\ntopics}(\bPi)},\quad i=1,\dots,n,
  \end{EQA}
  where 
  $(\ev_1, \dots, \ev_n)$ is the canonical basis of $\mathbb{R}^n$.
  \begin{lemma}
  \label{lemma:rowFactorBound}
    Assume that $\bPi \in \RR^{\nsize \t p}$ is a rank \(\ntopics\)
    matrix, 
    and
    $\bX \in \RR^{\nsize \t p}$ is any  matrix such that
    \(\|\bX - \bPi\| \le \frac{1}{2} \lambda_{\ntopics}(\bPi)\). Let \(\empiricalLeftEigenvectors, \probLeftEigenvectors\) be the \(\nsize \t \ntopics\) matrices of left singular vectors
    corresponding to the top \(\ntopics\) singular values of \(\bX\) and \(\bPi\), respectively. Then, there exist an orthogonal matrix \(\orthMatrix\) and a constant $C>0$ such that, for any $i=1,\dots,n$,
    \begin{EQA}
    \label{eq:rowConcentr}
      \|\ev_{i}^{\T}(\empiricalLeftEigenvectors - \probLeftEigenvectors \orthMatrix)\|_{2}
      &\le&
      C \beta_{i}(\bX, \bPi).
    \end{EQA}
  \end{lemma}
  %
  Define now
  \begin{EQA}[c]
    \label{eq:errorDef}
    \beta(\bX, \bPi) = \max_{i =1, \dots, \nsize} \beta_{i}(\bX, \bPi).
  \end{EQA}
  We will need the following condition.
  \begin{assumption} 
    \label{cond:community memberships}For a constant $\bar C>0$ we have
    \begin{EQA}[c]
      \beta(\bX, \bPi) \leq \frac{\bar C}{\lambda_{1}(\bW) \kappa(\bW) {K}\sqrt{K}}.
    \end{EQA} 
  \end{assumption}
  Assumption~\ref{cond:community memberships} is satisfied with high probability for $\bX\sim \mathbb{P}_{\bPi}$ when the sample size $N$ is large enough (see Appendix~\ref{Appendix:proof-of-thm1}). 
  Under Assumption~\ref{cond:community memberships}, we can derive from Lemma~\ref{lemma:rowFactorBound} the following deterministic bound on the error of estimating the topic-document matrix by the SPOC algorithm.  
  \begin{lemma}
    \label{lemma:documenTopicMatrixBound}
    Let Assumptions~\ref{ass:anchorDoc} and~\ref{cond:community memberships} be satisfied with constant $\bar C$ small enough. Assume that $\bPi \in \RR^{\nsize \t p}$ is a rank \(\ntopics\)
    matrix, 
    and
    $\bX \in \RR^{\nsize \t p}$ is any  matrix such that
    \(\|\bX - \bPi\| \le \frac{1}{2} \lambda_{\ntopics}(\bPi)\). Then, {matrix $\hat{\bf{H}}$ is non-degenerate and} the preconditioned SPOC algorithm outputs matrix \(\documentTopicMatrixEstimate\) such that 
    \begin{EQA}[c]
    \min_{\bP\in \mathcal{P}} \|\documentTopicMatrixEstimate - \documentTopicMatrix \bP\|_F \leq C \ntopics^{1/2} \left \{ \lambda^2_{\max}(\documentTopicMatrix) \kappa(\bW) \errorAdjacency + \frac{\kappa(\bPi)\lambda_{1}(\documentTopicMatrix)\|\bX - \bPi\|}{\lambda_{\ntopics}(\bPi)}\right\},
    \end{EQA}
    where $\mathcal{P}$ denotes the set of all permutation matrices.
  \end{lemma}
  Inspection of the proof shows that, for this lemma to hold, it is enough to choose the constant $\bar C \le \min(C_*,C_0^{-1})$ where 
  $C_*,C_0$ are the constants from Theorem~\ref{theorem:spaBasic} and Corollary~\ref{corollary:consistencyBasis}.

\subsection{Bounds with high probability}\label{subsec:bounds with high}
  Lemma~\ref{lemma:documenTopicMatrixBound} combined with a concentration inequality for $\|\bX - \bPi\|$ (cf. Lemma~\ref{lemma:spectralNormNoise} in the Appendix) allows us to derive a  bound for the estimation error that holds with high probability when $\bX$ is sampled from distribution $\mathbb{P}_{\bPi}$. Introduce the value
  \begin{EQA}[c]
    \Delta(\bW, \bPi) = \left(\dfrac{\lambda_{1}(\documentTopicMatrix)}{\lambda_{\ntopics}(\bPi)}\right)^2 \kappa(\bW) \kappa^2(\bPi).
  \end{EQA}
  The main result is summarized in the next theorem.
  \begin{theorem}
  \label{theorem:mainBound}
    Let Assumption~\ref{ass:anchorDoc} hold, 
    and $N_i=N$ for $i=1,\dots,n$. Assume that $N \ge \log(n + p)$ and 
    \begin{equation}\label{cond:lambda}
       \lambda_K(\bPi) \ge \sqrt{\frac{10}{\bar C}}\, K\left(\dfrac{n \log(n + p)}{N}\right)^{1/4} {\kappa(\bPi)\sqrt{\lambda_{1}(\bW)
       \kappa(\bW)}}. 
    \end{equation}
    Then,  with probability at least $1 - 2(n + p)^{-1}$, matrix $\hat{\bf{H}}$ is non-degenerate and the output \(\documentTopicMatrixEstimate\) of preconditioned SPOC algorithm satisfies, for some  constant ${\rm C}_1 >0$,  
    \begin{EQA}
      \min_{\bP\in \mathcal{P}}\bigl\|\documentTopicMatrixEstimate - \documentTopicMatrix \bP\bigr\|_{F}
      &\le&
      {\rm C}_1  \ntopics\sqrt{\dfrac{n \log(n+ p)}{N}} \Delta(\bW, \bPi), 
    \end{EQA} 
    where $\mathcal{P}$ denotes the set of all permutation matrices. 
  \end{theorem}

{Condition \eqref{cond:lambda} in Theorem~\ref{theorem:mainBound} guarantees that Assumption~\ref{cond:community memberships}  is satisfied. This condition holds if $N$ is large enough and it quantifies the separation of the spectrum of matrix $\bPi$ from zero.}  

The bound of Theorem~\ref{theorem:mainBound} depends on the singular values of matrices \(\documentTopicMatrix\) and \(\bPi\). 
  We now further detail this bound for the balanced case where matrices $\bW$ and $\bPi$ are well conditioned and the smallest non-zero singular value of $\bPi$ is of the same order as the largest singular value of $\bW$. It follows from Lemma~\ref{lemma:matrix_eigen} in the Appendix that in this case both $\lambda_{\ntopics}(\bPi)$ and  $\lambda_{1}(\bW)$ are of the order of $\sqrt{n/K}$. This is coherent with the behavior of the singular values of $\bPi$  and $\bW$ that we observed in the simulation study (see Section~\ref{simulations_sv}). The balanced case is formally described by the following assumption.
  \begin{assumption}\label{assumption_eigenvalues}
    There exist two constants \(C > 1\) and $c > 0$  
    such that
    \begin{equation*}
      \lambda_{\ntopics}(\bPi) \ge C \lambda_{1}(\bW) \quad
      \text{and} 
      \quad \max\left\{\kappa(\bPi), \kappa(\bW)\right\} \le c.
    \end{equation*}
  \end{assumption}
The second condition in Assumption \ref{assumption_eigenvalues} is quite standard and just states that matrices $\bPi$ and $\bW$ are well-conditioned. The first condition is more restrictive.  It  holds, in particular, if matrix $\bA$ is well-conditioned with large enough singular value $\lambda_{\ntopics}(\bA) $. For example, it will be the case  if $\bA$ satisfies the {\it anchor word assumption} (see Section \ref{sec:intro}) with the probabilities of anchor words uniformly above the probabilities of other words. This is detailed in Lemma \ref{lem:anchor_word} of the Appendix. {Noteworthy, the lower bound of Theorem \ref{theorem:lower_bound} below is attained with such choice of matrix $\bA$, see the proof of Theorem \ref{theorem:lower_bound} in Appendix \ref{sec:proof-of-lower-bound}. We can interpret it as the fact that, in a minimax sense, such matrices $\bA$ are associated
with the least favorable models.} 

The following corollary quantifies the behavior of SPOC estimator in the balanced case. 
  \begin{corollary}[Upper bound in the balanced case]
  \label{corollary:mainBound}
    Let  Assumptions~\ref{ass:anchorDoc} and~\ref{assumption_eigenvalues} hold, and $N_i=N$ for $i=1,\dots,n$. 
    Let also 
    \begin{align}\label{eq:condition-on-N}
       N & \ge C K^5
       \log(n + p)
    \end{align}
    for some ${C}>0$ large enough.
    Then,  with probability at least $1 - 2(n + p)^{-1}$, matrix $\hat{\bf{H}}$ is non-degenerate and the output \(\documentTopicMatrixEstimate\) of preconditioned SPOC algorithm satisfies, for some  constant ${\rm C}_2 >0$,  
    \begin{EQA}[c]
      \min_{\bP\in \mathcal{P}}\bigl\|\documentTopicMatrixEstimate - \documentTopicMatrix \bP\bigr\|_{F}
      \le
      {\rm C}_2  \ntopics \sqrt{\dfrac{n \log(n+ p)}{N}},
    \end{EQA} 
    where $\mathcal{P}$ denotes the set of all permutation matrices.
  \end{corollary}
  To prove Corollary~\ref{corollary:mainBound}, it is enough to notice that under Assumption~\ref{assumption_eigenvalues} we have $\Delta(\bW, \bPi)\le C'$ for some constant $C'>0$, and condition~\eqref{cond:lambda} follows from~\eqref{eq:condition-on-N}, Assumption~\ref{assumption_eigenvalues} and the inequality $\lambda_{1}(\bW)\ge \sqrt{n/K}$ (see Lemma~\ref{lemma:matrix_eigen} in the Appendix).
  
  Note that from Theorem~\ref{theorem:mainBound} and Corollary~\ref{corollary:mainBound} we can derive bounds in other norms. Thus, using the inequalities { $\bigl\|\documentTopicMatrixEstimate - \documentTopicMatrix \bP\bigr\|_{1}\leq \sqrt{Kn} \bigl\|\documentTopicMatrixEstimate - \documentTopicMatrix \bP\bigr\|_{F}$} and $\bigl\|\documentTopicMatrixEstimate - \documentTopicMatrix \bP\bigr\|_{1, \infty}\leq \sqrt{K} \bigl\|\documentTopicMatrixEstimate - \documentTopicMatrix \bP\bigr\|_{F}$, we obtain the following corollary:
  \begin{corollary}
  \label{corollary:one-infty-norm}
    If the assumptions of Theorem~\ref{theorem:mainBound} are satisfied then, 
    with probability at least $1 - 2(n + p)^{-1}$, matrix $\hat{\bf{H}}$ is non-degenerate and the output \(\documentTopicMatrixEstimate\) of preconditioned SPOC algorithm satisfies
    \begin{EQA}
      \min_{\bP\in \mathcal{P}}\bigl\|\documentTopicMatrixEstimate - \documentTopicMatrix \bP\bigr\|_{1,\infty}
      &\le&
      {\rm C}_1 \ntopics^{3/2}\sqrt{\dfrac{n \log(n+ p)}{N}} \Delta(\bW, \bPi)\quad \text{and}\\ \min_{\bP\in \mathcal{P}}\bigl\|\documentTopicMatrixEstimate - \documentTopicMatrix \bP\bigr\|_{1}
      &\le&
      {\rm C}_1 \ntopics^{3/2}n\sqrt{\dfrac{ \log(n+ p)}{N}} \Delta(\bW, \bPi). 
    \end{EQA} 
    If the assumptions of Corollary~\ref{corollary:mainBound} are satisfied then,  with probability at least $1 - 2(n + p)^{-1}$, matrix $\hat{\bf{H}}$ is non-degenerate and the output \(\documentTopicMatrixEstimate\) of preconditioned SPOC algorithm satisfies
    \begin{EQA}[c]
      \label{eq:vtor}
      \min_{\mathbf{P}\in \mathcal{P}}
    \bigl\|\documentTopicMatrixEstimate - \documentTopicMatrix \bP\bigr\|_{1,\infty}
    \le
    {\rm C}_2 \ntopics^{3/2} \sqrt{\dfrac{n \log(n+p)}{N}}\quad \text{and}\quad \min_{\mathbf{P}\in \mathcal{P}}
    \bigl\|\documentTopicMatrixEstimate - \documentTopicMatrix \bP\bigr\|_{1}
    \le
    {\rm C}_2 \ntopics^{3/2}n \sqrt{\dfrac{ \log(n+p)}{N}}
    \end{EQA} 
  \end{corollary}
   It follows from Corollaries \ref{corollary:mainBound} and 
   \ref{corollary:one-infty-norm} that, for all the considered norms, the rate of estimating $\bW$ (to within a weak factor) is determined by two parameters, which are the number of documents $n$ and the sample size $N$. The dependency on the size of the dictionary $p$ is weak. This is confirmed by the numerical experiments, see Section \ref{sec:simulations}.

 \subsection{Adaptive procedure when $K$ is unknown.} 
 
 We now propose an adaptive variant of the SPOC algorithm when the number of topics $K$ is unknown. It is obtained by replacing $K$ in Algorithm \ref{algo:SPOC} by the estimator
 $$
 \hat K = \max\Big\{j: \ \lambda_j(\bX)>4 \sqrt{\dfrac{n \log(n+p)}{N}} \Big\}.
 $$
 In the sequel, the resulting procedure will be called the adaptive (preconditioned) SPOC algorithm. The following analogs of Theorem~\ref{theorem:mainBound} and Corollary~\ref{corollary:mainBound} hold.
 \begin{theorem}
  \label{theorem:mainBound-adapt}
    Let the assumptions of Theorem~\ref{theorem:mainBound} be satisfied and 
    \begin{equation}\label{cond:lambda-adapt}
       \lambda_{1}(\bW) > \frac{32 \bar C}{5 K^2}\, \sqrt{\dfrac{n \log(n+p)}{N}} . 
    \end{equation}
    Then, with probability at least $1 - 2(n + p)^{-1}$, matrix $\hat{\bf{H}}$ is non-degenerate, {$\hat K=K$,} and the output \(\documentTopicMatrixEstimate\) of the adaptive preconditioned SPOC algorithm satisfies
    \begin{EQA}
      \min_{\bP\in \mathcal{P}}\bigl\|\documentTopicMatrixEstimate - \documentTopicMatrix \bP\bigr\|_{F}
      &\le&
      {\rm C}_1 \ntopics \sqrt{\dfrac{n \log(n+ p)}{N}} \Delta(\bW, \bPi). 
    \end{EQA}
    \end{theorem}
    \begin{corollary}
  \label{corollary:mainBound-adapt}
    Let the assumptions of Corollary~\ref{corollary:mainBound} and \eqref{cond:lambda-adapt} be satisfied. Then,  with probability at least $1 - 2(n + p)^{-1}$, matrix $\hat{\bf{H}}$ is non-degenerate, {$\hat K=K$,} and the output \(\documentTopicMatrixEstimate\) of the adaptive preconditioned SPOC algorithm satisfies
    \begin{EQA}[c]
      \label{eq:vtor}
      \min_{\mathbf{P}\in \mathcal{P}}
    \bigl\|\documentTopicMatrixEstimate - \documentTopicMatrix \bP\bigr\|_{F}
    \le
    {\rm C}_2 \ntopics \sqrt{\dfrac{n \log(n+p)}{N}}.
    \end{EQA} 
  \end{corollary}
 Note that condition \eqref{cond:lambda-adapt} introduced in Theorem~\ref{theorem:mainBound-adapt} and Corollary~\ref{corollary:mainBound-adapt} additionally to the conditions of Theorem~\ref{theorem:mainBound} and Corollary~\ref{corollary:mainBound} is rather mild. Indeed, due to inequality  \eqref{eq:matrix_property2} proved in the Appendix we have $\lambda_{1}(\bW)\ge \sqrt{n/K}$.
 Therefore, it is sufficient that $N> \frac{C\log(n+p)}{K^3}$ to grant \eqref{cond:lambda-adapt}.

\subsection{Minimax lower bound} 
  The following lower bound shows that the rate obtained in Corollary~\ref{corollary:mainBound} is near minimax optimal.  Denote by $\mathcal{M}$ the class of all matrices $\bPi$ satisfying the assumptions stated in the Section \ref{sec:intro} and Assumption~\ref{assumption_eigenvalues}. 
  \begin{theorem}[Lower bound]
  \label{theorem:lower_bound}
    Assume that  $N_i=N$ for $i=1,\dots,n$ and $2\leq K\leq \min( p/4,N/2,{n/2})$. 
    Then, there exist two constants $C >0$ and $c \in ( 0,1)$ such that, for any estimator and ${\overline{\bW}}$  of $\bW$ we have
    \begin{align}\label{inf_bnd_l2}
       \sup_{\bPi\in \mathcal{M}} \mathbb{P}_{\bPi} 
      \bigg\{\min_{\mathbf{P}\in \mathcal{P}}\|\overline{\bW} - \bW\mathbf{P}\|_F \geq C \sqrt{\frac{n}{N}} \bigg\} 
      \geq c,
    \end{align}
    and
 \begin{align}\label {inf_bnd_l1}
       \sup_{\bPi\in \mathcal{M}} \mathbb{P}_{\bPi} 
      \bigg\{\min_{\mathbf{P}\in \mathcal{P}}\|\overline{\bW} - \bW\mathbf{P}\|_1 \geq Cn \sqrt{\frac{K}{N}} \bigg\} 
      \geq c,
    \end{align}
    where $\mathcal{P}$ denotes the set of all permutation matrices. 
  \end{theorem}
 Combining Corollary~\ref{corollary:mainBound} and \eqref{inf_bnd_l2}  we find that, to within a weak factor, the minimax optimal rate of estimation of $\bW$ on the class $\mathcal{M}$ in the Frobenius norm scales as $\sqrt{n/N}$.  On the other hand,  \eqref{inf_bnd_l1} and Corollary~\ref{corollary:one-infty-norm} imply that, again to within a weak factor, the minimax optimal rate of estimation of $\bW$ in the $l_1$-norm on $\mathcal{M}$ scales as $n/\sqrt{N}$. 
 
 {
 \begin{remark}
 \label{rem1}
 Inspection of the proof of Theorem \ref{theorem:lower_bound} shows that the lower bound is in fact established for a subset of $\mathcal{M}$ composed of matrices satisfying both anchor word and anchor document assumptions. 
 \end{remark}
 \begin{remark}\label{rem2}
 It is proved in~\cite{Ke2017,bing2018fast} that, under the same model of observations and anchor word assumption, the minimax optimal rate for estimation of matrix $\bA$ in the $\ell_1$-norm 
  scales as $\sqrt{\frac{p}{nN}}$ (to within weak factors). 
  Note that this rate is determined by all the three main parameters of the problem - the size of the dictionary $p$, the number of documents $n$ and the sample size $N$.  
  This is quite different from the minimax $\ell_1$-rate $n/\sqrt{N}$ of estimation $\bW$, which remains valid under anchor word assumption, cf. Remark \ref{rem1} and the remark after Assumption \ref{assumption_eigenvalues}. 
  It shows that there is a significant difference between the problems of estimating matrices $\bA$ and $\bW$ in topic models. 
 \end{remark}
 
 }

\section{Related Work}
\label{sec:comparison}
  There exists an extensive literature on topic modeling and several algorithms have been proposed for estimation of matrices $\bA$ and $\bW$. As the problem of recovering these two matrices  when there is no noise is an instance of non-negative matrix factorisation several papers propose algorithms based on minimization of  a regularized cost function, see, e.g., \cite{lee99,Donoho2004,Cichocki2009,Recht2012}. Such methods result in non-convex optimization and often fail when many words do not appear in a single document, that is when $N\ll p$.

  Another approach is to use Bayesian methods such as the popular \textit{Latent Dirichlet Allocation} (LDA) introduced in~\cite{blei2003latent}.
  LDA proceeds by imposing a Dirichlet prior on $\bA$ and then computing an estimator of $\bW$ by a variational EM-algorithm. The original paper~\cite{blei2003latent} and the subsequent line of work do not provide statistical guarantee on the recovery of $\bW$. In~\cite{blei2003latent}, the authors argue that {LDA} avoids two issues of the {pLSI} that are the risk of overfitting and the difficulty of classifying a new document outside the corpus. Yet, {LDA} is computationally slow and makes the assumption that topics are uncorrelated, which may be not realistic~\cite{blei2007correlated,li2006pachinko}. This last issue has been addressed in~\cite{lafferty2006correlated} by introducing \textit{Correlated topic models}. LDA has been extended to relax some  assumptions such as the \textit{bag-of-words} hypothesis (``order of words does not matter'')~\cite{wallach2006topic}, the exchangeability of documents (``topics do not vary in time'')~\cite{blei2006dynamic}, the assumption that the number of topics is known~\cite{teh2005sharing}. Also, to recover $\bW$ in the LDA setting, some papers used Gibbs-sampling~\cite{ramage2009labeled, porteous2008fast} or variational Bayes techniques~\cite{zhai2012mr,chien2010dirichlet} rather than the EM-algorithm. However, these works do not provide statistical guarantees on the estimation of $\bW$ and the associated algorithms are computationally slow.
 
 {
  For the problem of estimation of matrix $\bA$, papers~\cite{arora2012learning,anandkumar2012spectral,arora2013practical,Ding2013,anandkumar2014tensor,bansal2014provable,Ke2017,bing2018fast}, to mention but a few, provided algorithms with provable statistical guarantees under {\it anchor word assumption}. They proposed various techniques based, for example, on analyzing co-occurrence matrices, tensors, or on recovering vertices of a simplex using SVD. Most of these papers, except for~\cite{Ke2017, bing2018fast}, do not work under the same statistical model as ours (cf. Section \ref{sec:intro}). Thus,  \cite{arora2012learning,arora2013practical,Ding2013} assume that topic-document matrix $\bW$ is randomly generated from some prior distribution. For a setting with no randomness, \cite{Mizutani14} proposes ellipsoidal rounding algorithm with application to topic models.
  Paper~\cite{mao2018overlapping} develops a generalized method to bind overlapping clustering models, including topic models. Moreover, for some classes of matrices, papers~\cite{Ke2017, bing2018fast} proposed minimax optimal algorithms of estimating $\bA$. Both~\cite{Ke2017, bing2018fast} impose {\it anchor word assumption} but their estimators are different. Thus, ~\cite{Ke2017} performs SVD on properly  normalized matrix $\bX$ followed by an exhaustive search over a $p$-dimensional simplex, while ~\cite{bing2018fast} proceeds by first recovering the anchor words and then deriving estimators of $\bA$ from a scaled version of matrix $\bX \bX^\T$. 
  }

\section{Simulations}
\label{sec:simulations}
\subsection{Synthetic Data}
  We first present the results of experiments on synthetic data. 
  We have performed simulations with different values of the  parameters $n,p,N$ and the number of topics $K$. Our objective was to observe the effect of each parameter on the Frobenius error between $\bW$ and its estimator $\hat{\bW}$ obtained by the SPOC  algorithm. We report the results for the SPOC algorithm without preconditioning step as it had a negligible impact on the performance of the method while being computationally demanding. As a benchmark, we use the LDA algorithm~\cite{blei2003latent}. For the experiments we use the Python implementation of SPOC\footnote{The code of SPOC algorithm is available at https://github.com/stat-ml/SPOC} and an implementation of the LDA algorithm available in Sklearn~\cite{scikit-learn}.

  Figures~\ref{fig:fig_n_1}-\ref{fig:fig_K_1} present an example of results that we have typically obtained in simulations. In Figures~\ref{fig:fig_n_1}-\ref{fig:fig_p_1} we take $K=3$ and the matrix $\bW$ has the following structure: $K$ rows of $\bW$ are canonical basis vectors, each of the remaining $n-K$ rows is generated independently using the Dirichlet distribution with parameter $\alpha = (0.1,0.15,0.2)$. In Figure~\ref{fig:fig_K_1}, where $K$ must vary, we define $\bW$ in a different way. Namely, for the $n-K$ rows that are not canonical basis vectors, each element \(W_{kj}\) is generated from the uniform distribution on $[0, 1]$ and then each row  of the matrix is normalized so as to have \(\sum_{k = 1}^K W_{ik} = 1\). For the matrix $\bA$, we take $K$ columns proportional to canonical basis vectors. The elements \(A_{kj}\) of matrix \(\bA\) in the remaining $p-K$ columns are obtained by generating numbers from the uniform distribution on $[0, 1]$ and then normalizing each row of the matrix to have \(\sum_{j = 1}^p A_{kj} = 1\). For given $\bW$ and $\bA$, the data matrix $\bX$ is generated according to the pLSI model  defined in Section~\ref{sec:intro}.
  For each value on the $x$-axes of the figures, we present the averaged result over 10 simulations. 
  
  We clearly retrieve the patterns indicated in Theorem~\ref{theorem:mainBound}, Corollary~\ref{corollary:mainBound} and \eqref{inf_bnd_l2}. Thus, the plots have a near \(\sqrt{n}\) behaviour in Figure~\ref{fig:fig_n_1} and a near \(1 / \sqrt{N}\) behaviour in Figure~\ref{fig:fig_N_1}. Figure~\ref{fig:fig_p_1} shows  weak dependence of the error of the SPOC algorithm on the size of the dictionary \(p\), which agrees with the bound obtained in Corollary~\ref{corollary:mainBound}. 
  {This can be interpreted as one of the advantages of our method over LDA. Indeed, for LDA we observe that the error increases significantly as $p$ grows.} Finally, we notice the linear dependence of the error on \(K\) as predicted by Corollary~\ref{corollary:mainBound}. 
  In all the experiments we observe that SPOC algorithm is very competitive with LDA while being much more stable.

  \begin{figure}[t]
    \centering
    \includegraphics[scale=0.3]{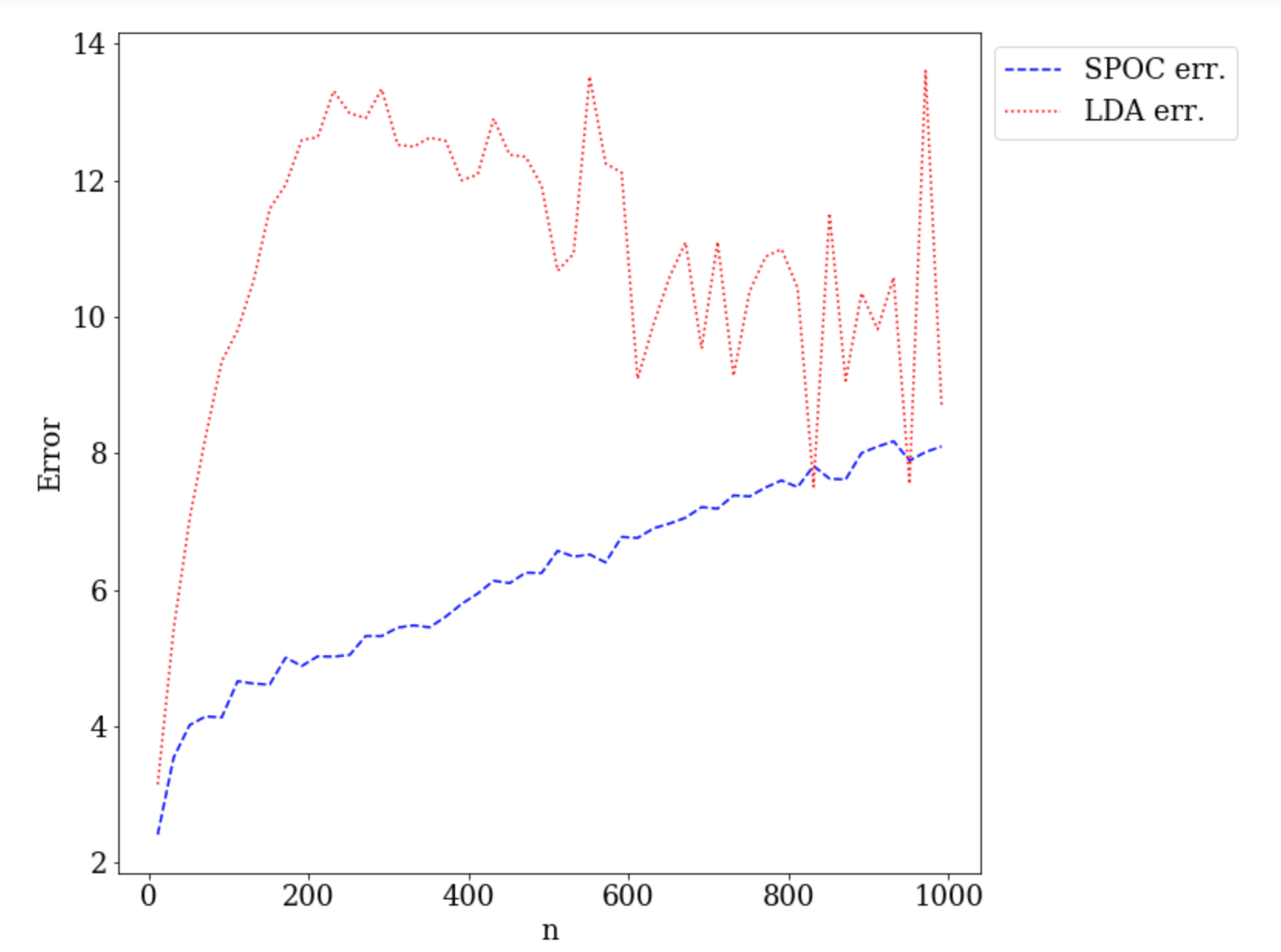}
    \caption{The $n$-dependency of  {$\min_{\mathbf{P}\in \mathcal{P}}\|\bW - \hat{\bW}\mathbf{P}\|_F$} using SPOC and LDA algorithms.
    Total number of words $p=5000$, number of sampled words in each document $N=200$. 
    }
  \label{fig:fig_n_1}
  \end{figure}
   
  \begin{figure}[H]
    \centering
    \includegraphics[scale=0.3]{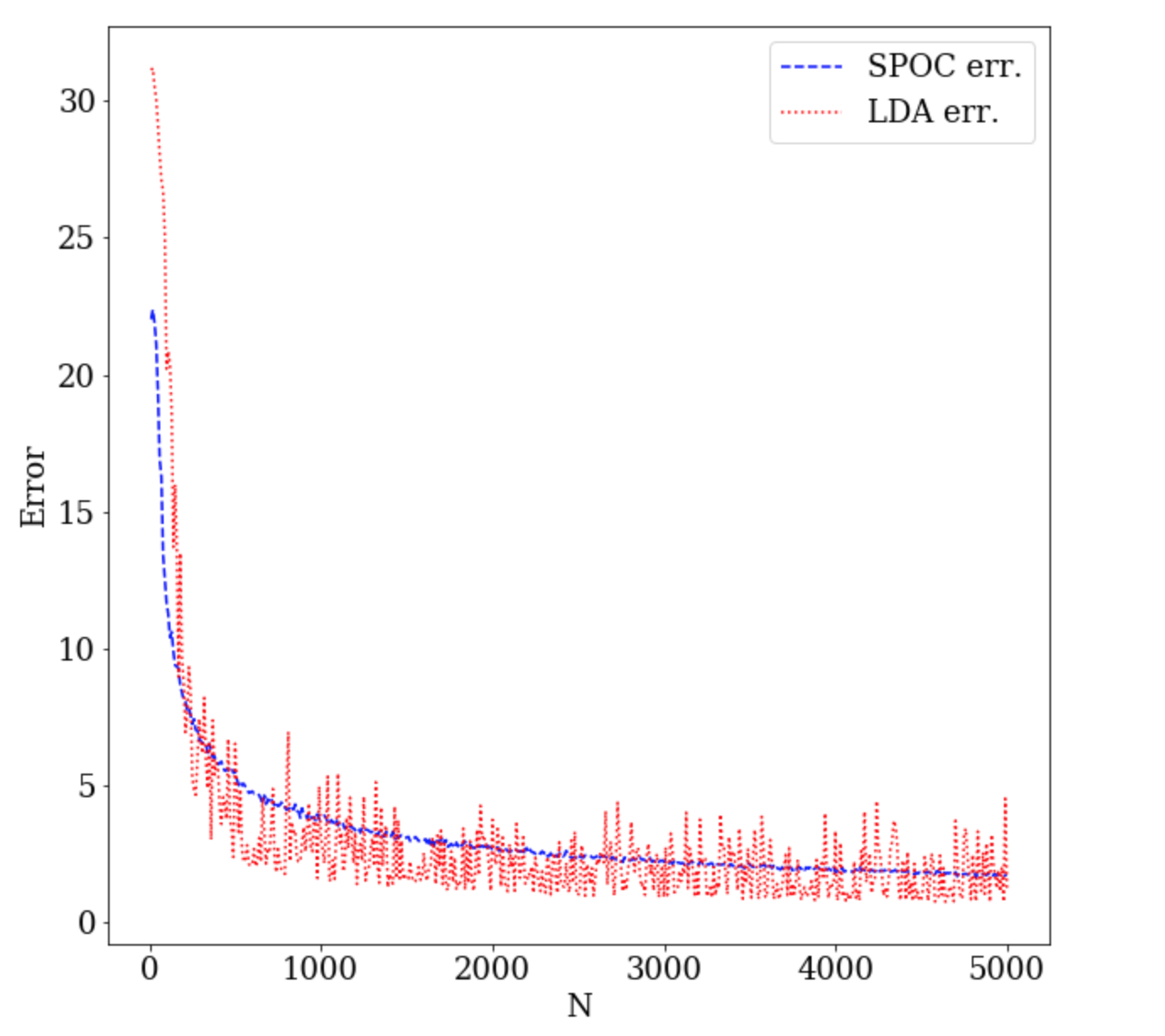}
    \caption{The $N$-dependency of {$\min_{\mathbf{P}\in \mathcal{P}}\|\bW - \hat{\bW}\mathbf{P}\|_F$} using SPOC and LDA algorithms.
    Total number of words $p=5000$, number of documents $n=1000$.
    }
  \label{fig:fig_N_1}
  \end{figure}

  \begin{figure}[H]
    \centering
    \includegraphics[scale=0.3]{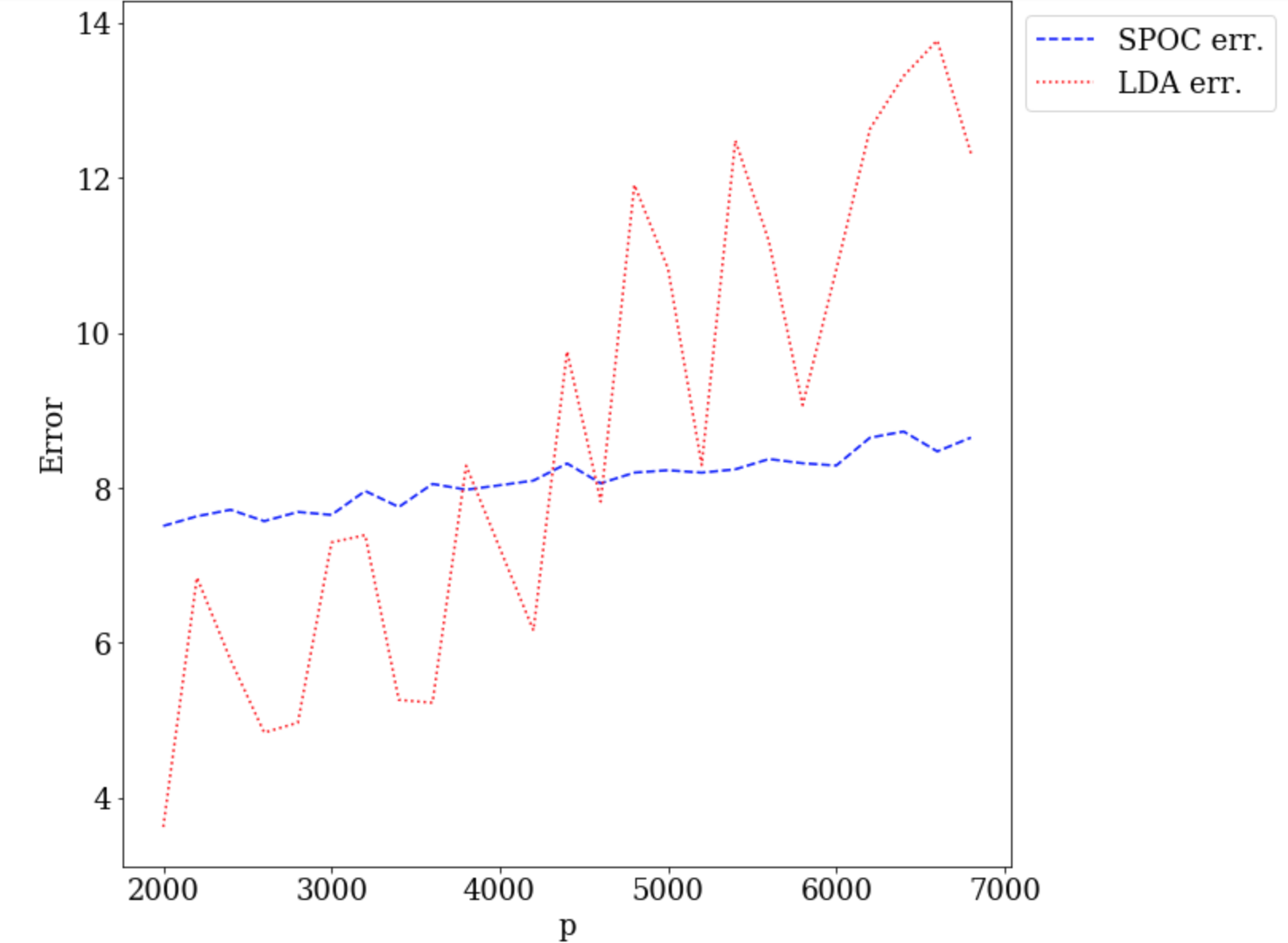}
    \caption{The $p$-dependency of  {$\min_{\mathbf{P}\in \mathcal{P}}\|\bW - \hat{\bW}\mathbf{P}\|_F$} using SPOC and LDA algorithms.
    {Number of documents} $n=1000$, number of sampled words in each document $N=200$.
    }
  \label{fig:fig_p_1}
  \end{figure}

  \begin{figure}[H]
    \centering
    \includegraphics[scale=0.3]{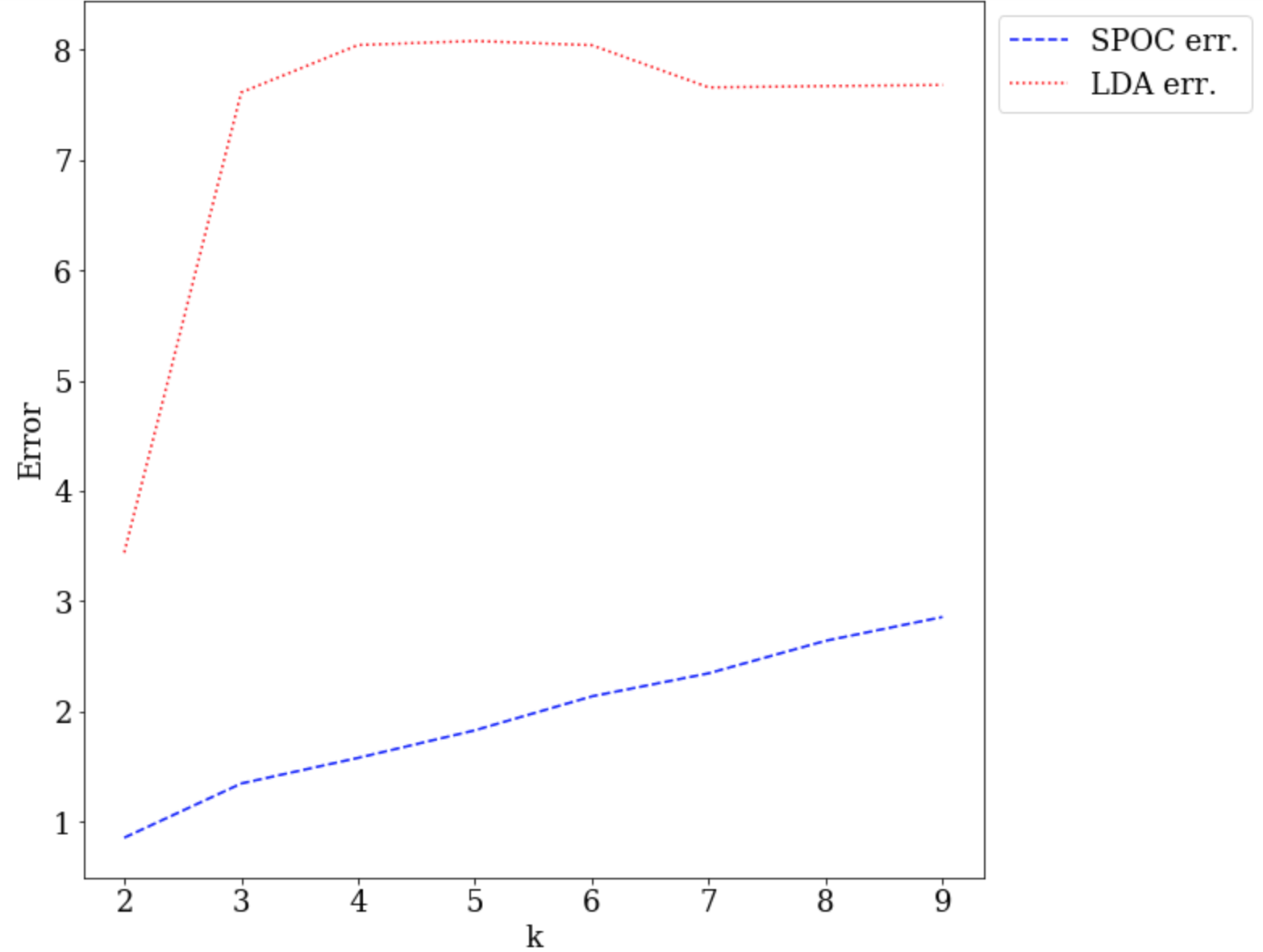}
    \caption{The $K$-dependency of {$\min_{\mathbf{P}\in \mathcal{P}}\|\bW - \hat{\bW}\mathbf{P}\|_F$} using SPOC and LDA algorithms.
    Total number of words $p=5000$, number of sampled words in each document $N=5000$.
    }
  \label{fig:fig_K_1}
  \end{figure}

  A numerical study of the SPOC estimator $\hat {\bA}$ of matrix $\bA$  is deferred to Section~\ref{sec:topic_word_experiments} of the Appendix. It shows that $\hat {\bA}$ behaves similarly to the corresponding LDA estimator while being more stable.

\subsection{Corpus of NIPS abstracts}
  We now illustrate the performance of our algorithm applying it to the data set  
  of full texts of NIPS papers\footnote{The link to the dataset: https://archive.ics.uci.edu/ml/datasets/NIPS+Conference+Papers+1987-2015}~\cite{Perrone2016Nips}.
  This data set contains the distribution of words in the full text of the NIPS conference papers published from 1987 to 2015. The data set has the form of a $11463 \times 5811$ matrix of word counts containing $11463$ words and $5811$ NIPS conference papers. Each column contains the number of times each word appears in the corresponding document.

  We start by pre-processing the data. We first remove all the documents with less than 150 words. Then we remove from the resulting dictionary the stop words and the words that appear in less that 150 documents. This results in a database of $5801$ documents with a dictionary of $6380$ words.
  Note that we do not actually have access to the true topics of each document. We also do not have access to the true number of topics. In order to compare our method to LDA,  we proceeded as follows. For each value of $K = 3, \dots, 10$, we first computed the LDA estimator $\tilde{\bW}$ of the document-topic matrix  and the LDA estimator $\tilde{\bA}$ of the topic-word matrix. Next, with the underlying matrix $\tilde{\bPi}=\tilde{\bW}\tilde{\bA}$, for each value of $K$ we simulated 10 matrices $\tilde{\bX}$ with $N=200$ sampled words {according to pLSI model}. For each matrix $\tilde{\bX}$, we estimated $\tilde{\bW}$ using both LDA and  SPOC algorithms. Finally, for each $K$ we computed the mean error over 10 simulations. The resulting comparison as function of $K$ is presented in Figure~\ref{fig:fig_K_2}. We can observe that SPOC systematically outperforms the LDA algorithm, except for $K = 2$.
  
  \begin{figure}[H]
    \centering
    \includegraphics[scale=0.3]{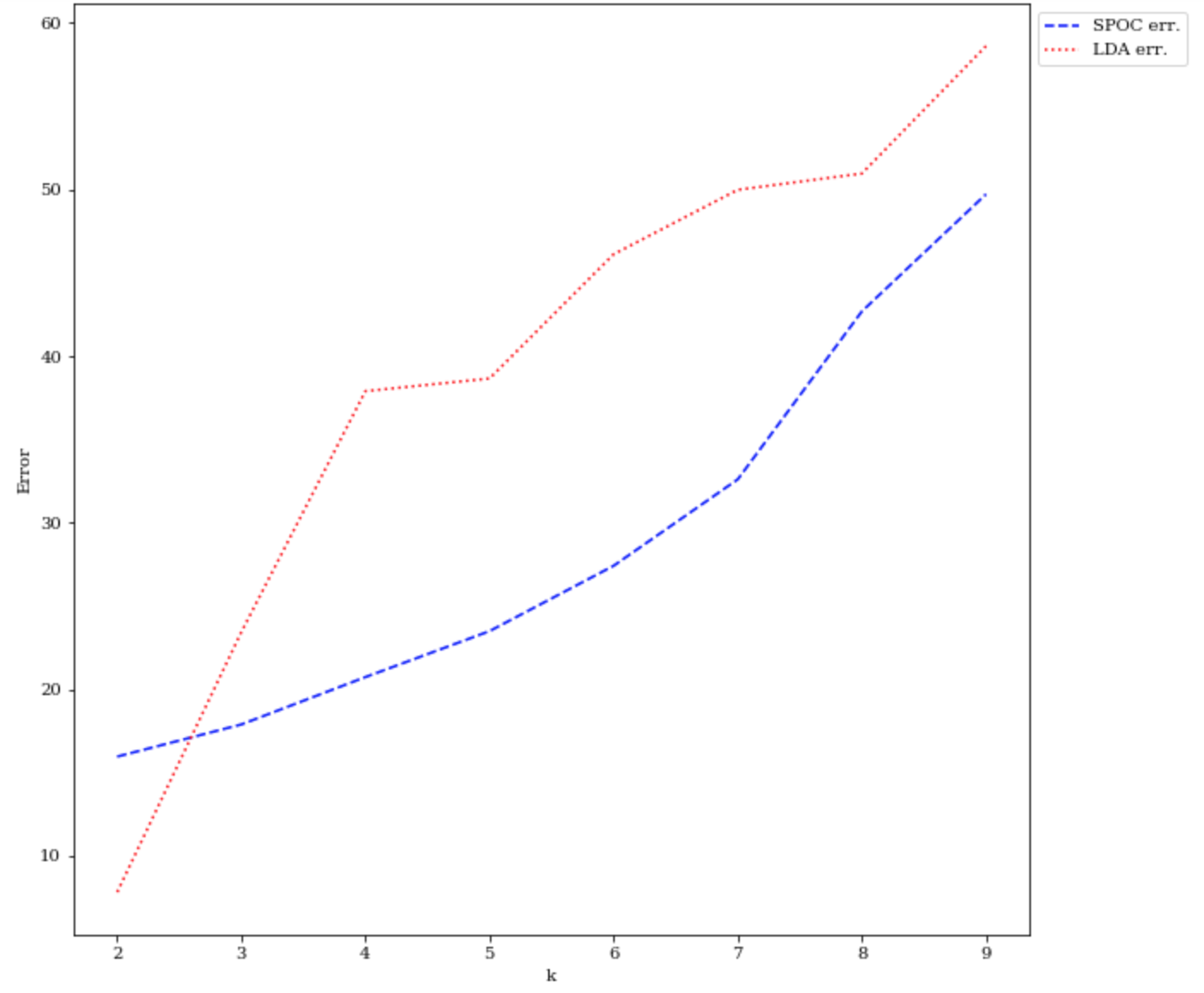}
    \caption{The $K$-dependency of  {$\min_{\mathbf{P}\in \mathcal{P}}\|\bW - \hat{\bW}\mathbf{P}\|_F$}  using SPOC and LDA algorithms on semi-synthetic data. Matrix $\tilde{\bW}$ is the LDA estimator on the NIPS data set ($n=5081$ documents, $p=6380$ words), and $\hat{\bW}$ is the LDA or SPOC estimator on data simulated from
    $\tilde{\bPi}$. 
    }
  \label{fig:fig_K_2}
  \end{figure}

  Next, we investigate whether our estimator of matrix $\bW$ helps to well classify the documents in the NIPS corpus. We apply the simplest possible classifier based on the obtained SPOC estimator~$\hat{\bW}$, namely, we classify article $i$ to the topic that has the maximum value of $\hat{W}_{ik}$ for $k = 1, \dots, K$. We consider the number of topics $K = 3$. While apparently for NIPS papers some words (such as ``learning'' or ``model'') are more frequent than others in the full corpus, other words can be more frequent for particular topics. Therefore, for each topic we present the words that have the highest difference {between their frequency for this topic and their maximum frequency in other topics}. We clearly see that the obtained topics are semantically well separated, cf. Table~\ref{tab:topics}.

  \begin{table}[H]
    \centering

    \caption{Top 10 words, which have the highest difference in frequency for each topic compared to other topics. The three topics were identified by SPOC method.}
    \begin{tabular}{|c|c|c|c|}
      \hline
      & ``Neural networks'' & ``{Statistical} 
      learning'' & ``Algorithms and theory''
      \\
      \hline
      1 & network & model & algorithm
      \\
      \hline
      2 & input & data & learning
      \\
      \hline
      3 & neural & image & function
      \\
      \hline
      4 & neurons & distribution & problem
      \\
      \hline
      5 & units & inference & set
      \\
      \hline
      6 & output & likelihood & theorem
      \\
      \hline
      7 & layer & latent & bound
      \\
      \hline
      8 & neuron & prior & matrix
      \\
      \hline
      9 & system & Gaussian & loss
      \\
      \hline
      10 & synaptic & parameters & error
      \\
      \hline
    \end{tabular}
  \label{tab:topics}
  \end{table}

\section{Conclusion}
\label{sec:discussion}
  In the present paper, we proposed the SPOC algorithm, which is a computationally efficient procedure to estimate the document-topic matrix in topic models with known or unknown number of topics $K$. It is based on the Successive Projection Algorithm used to recover the vertices of a $K$-dimensional simplex in the context of separable matrix factorization. We developed the statistical analysis of SPOC algorithm under the \textit{anchor document assumption} requiring that, for each topic, there is a document devoted solely to this topic. We proved that the proposed method is near minimax optimal for estimation of the document-topic matrix under the Frobenius norm and the $\ell_1$-norm. 
  As an element of our analysis, we derived  a bound on  concentration of  matrices with independent multinomial columns that may be of independent interest. The theoretical results are supported by empirical evidence demonstrating a good performance of the SPOC algorithm  and its advantages compared to LDA.
  
  \medskip
  
  {\bf Acknowledgement.} The research of Suzanne Sigalla and Alexandre B. Tsybakov is supported by the grant of French National Research Agency (ANR) "Investissements d'Avenir" LabEx Ecodec/ANR-11-LABX-0047.


\appendix

\section{Tools}
\label{sec:tools}
\subsection{Matrix Perturbation Bounds}
\label{subsec:perturbation}
  In this section, we provide some facts about matrix perturbation that will be used in the proofs. We start with the following lemma, which is a variant of Davis-Kahan theorem.
  \begin{proposition}[Lemma 5.1 in~\cite{Lei2015}]
  \label{prop:Davis-Kahan}
   Let \({\bf M} \in \RR^{\nsize \t \nsize}\) be a rank \(K\) symmetric matrix with smallest nonzero eigenvalue \(\lambda_{K}({\bf M})\), and let \(\hat{\bf M}\in \RR^{\nsize \t \nsize}\) be any symmetric matrix. Let $\empiricalLeftEigenvectors(\hat{\bf M})\in \RR^{\nsize \t K}$ and $\probLeftEigenvectors ({\bf M})\in \RR^{\nsize \t K}$ be the matrices of \(K\) leading  eigenvectors of \(\hat{\bf M}\) and \({\bf M}\), respectively. Then there exists a \(K \t K\) orthogonal matrix \(\orthMatrix\) such that
   \begin{EQA}[c]
     \|\empiricalLeftEigenvectors(\hat{\bf M}) - \probLeftEigenvectors({\bf M}) \orthMatrix\|_{F} \le \frac{2 \sqrt{2 K} \|\hat{\bf M} - {\bf M}\|}{\lambda_{K}({\bf M})}.
   \end{EQA}
  \end{proposition}

  \begin{corollary}\label{lemma:eigenvectorConsistency}
    Let $\bPi$ and $\bX$ be matrices with singular value decompositions given by~\eqref{svc:Pi} and~\eqref{svc:X}. Then there exist  \(\ntopics \t \ntopics\) orthogonal matrices \(\orthMatrix\) and \(\widetilde{\orthMatrix}\) such that
    \begin{EQA}[c]\label{DKSingularVectors}
      \|\empiricalLeftEigenvectors - \probLeftEigenvectors \orthMatrix\|_{F} 
      \le \frac{2\sqrt{2K}(\|\bX\| + \|\bPi\|)\|\bX - \bPi\|}{\lambda_{\ntopics}^2( \bPi)}
    \end{EQA}
    and
    \begin{EQA}[c]\label{DKSingularVectorsRight}
      \|\empiricalRightEigenvectors - \probRightEigenvectors \widetilde{\orthMatrix}\|_{F}
      \le 
      \frac{2\sqrt{2K}(\|\bX\| + \|\bPi\|)\|\bX - \bPi\|}{\lambda_{\ntopics}^2( \bPi)}.
    \end{EQA}
    Furthermore, if \(\|\bX - \bPi\| \le \frac{1}{2} \lambda_{\ntopics}(\bPi)\) then 
    \begin{EQA}[c]\label{DKSingularVectors-bis}
    \max \big(\|\empiricalLeftEigenvectors - \probLeftEigenvectors \orthMatrix\|_{F}, \|\empiricalRightEigenvectors - \probRightEigenvectors \widetilde{\orthMatrix}\|_{F}\big)
      \le 
      \frac{5\sqrt{2K}\kappa(\bPi)\|\bX - \bPi\|}{\lambda_{\ntopics}( \bPi)}.
    \end{EQA}
  \end{corollary}
  \begin{proof}
    Applying Proposition~\ref{prop:Davis-Kahan} to matrices $\bPi\bPi^\T$ and $\bX\bX^\T$ we get 
    \begin{EQA}[c]
      \|\empiricalLeftEigenvectors - \probLeftEigenvectors\orthMatrix\|_F\leq \dfrac{2\sqrt{2K}\|\bPi\bPi^\T-\bX\bX^\T\|}{\lambda_{\ntopics}(\bPi\bPi^\T)}
      \leq 
      \frac{2\sqrt{2K}(\|\bX\| + \|\bPi\|)\|\bX - \bPi\|}{\lambda_{\ntopics}^2( \bPi)}.
    \end{EQA}
    Similarly, inequality~\eqref{DKSingularVectorsRight} is obtained by applying Proposition~\ref{prop:Davis-Kahan}
    to matrices $\bPi^\T\bPi$ and $\bX^\T\bX$.
    Next, if  \(\|\bX - \bPi\| \le \frac{1}{2} \lambda_{\ntopics}(\bPi)\) then due to the triangle inequality
   we have $\|\bX\| \le \|\bPi\| + \frac{1}{2} \lambda_{\ntopics}(\bPi) \le \frac{3}{2}\|\bPi\|$. Combining this fact with~\eqref{DKSingularVectors} and~\eqref{DKSingularVectorsRight} we obtain~\eqref{DKSingularVectors-bis}. 
  \end{proof}
  We will also need the following bounds for matrices of singular values $\adjacencyEigenvalues$ and $\probEigenvalues$.
  \begin{lemma}
    \label{corollary:eigenvalues}
    Let the assumptions of Corollary~\ref{lemma:eigenvectorConsistency} hold. Let $\adjacencyEigenvalues$ and $\probEigenvalues$ be diagonal \(\ntopics \times \ntopics\)-matrices of \(\ntopics\) largest singular values of \(\bX\) and \(\bPi\), respectively, cf.~\eqref{svc:Pi} and~\eqref{svc:X}. If \(\|\bX - \bPi\| \le \frac{1}{2} \lambda_{\ntopics}(\bPi)\) then 
    \begin{EQA}[c]
      \|\adjacencyEigenvalues - \orthMatrix^{\T} \probEigenvalues \widetilde{\orthMatrix}\|
      \le
      C\kappa^2(\bPi)\sqrt{\ntopics} ~ \|\bX - \bPi\|
    \end{EQA}
    and
    \begin{EQA}[c]\label{DKSingularValues}
      \|\adjacencyEigenvalues^{-1} - \widetilde{\orthMatrix}^{\T} \probEigenvalues^{-1} \orthMatrix\|
      \le
      C\kappa^2(\bPi)\sqrt{\ntopics}\frac{\|\bX - \bPi\|}{ \lambda^2_{\ntopics}(\bPi)},
    \end{EQA}
    where the orthogonal matrices \(\orthMatrix\), \(\widetilde{\orthMatrix}\) are the same as in Corollary~\ref{lemma:eigenvectorConsistency}.
  \end{lemma}
  
  \begin{proof}
    
    Applying Weyl's inequality~\cite[Theorem C.6]{Giraud-book} we get 
    \begin{EQA}[c]
      \|\empiricalLeftEigenvectors \adjacencyEigenvalues \empiricalRightEigenvectors^{\T} - \probLeftEigenvectors \probEigenvalues \probRightEigenvectors^{\T}\| \le 2\|\bPi  - \bX\|
    \end{EQA}
    and further
    \begin{EQA}
      \|\empiricalLeftEigenvectors \adjacencyEigenvalues \empiricalRightEigenvectors^{\T} - \probLeftEigenvectors \probEigenvalues  \probRightEigenvectors^{\T}\|
      &\ge&
      \|\probLeftEigenvectors \orthMatrix (\adjacencyEigenvalues - \orthMatrix^{\T} \probEigenvalues \widetilde{\orthMatrix}) \empiricalRightEigenvectors^{\T}\|
      \\
      && -
      \|(\empiricalLeftEigenvectors - \probLeftEigenvectors \orthMatrix) \adjacencyEigenvalues \empiricalRightEigenvectors^{\T}\|
      -
      \|\probLeftEigenvectors \probEigenvalues \widetilde{\orthMatrix} (\empiricalRightEigenvectors - \probRightEigenvectors \widetilde{\orthMatrix})^{\T}\|.
    \end{EQA}
    Therefore
    \begin{EQA}
      \|\adjacencyEigenvalues - \orthMatrix^{\T} \probEigenvalues \widetilde{\orthMatrix}\|
      &\le&
      \|\bX - \bPi\|
      +
      \|(\empiricalLeftEigenvectors - \probLeftEigenvectors \orthMatrix) \adjacencyEigenvalues \empiricalRightEigenvectors^{\T}\|
      +
      \|\probLeftEigenvectors \probEigenvalues \widetilde{\orthMatrix} (\empiricalRightEigenvectors - \probRightEigenvectors \widetilde{\orthMatrix})^{\T}\|
      \\
      &\le&
      \|\bX - \bPi\| + \|\bX\|  \|\empiricalLeftEigenvectors - \probLeftEigenvectors \orthMatrix\|+ \|\bPi\|\|\empiricalRightEigenvectors - \probRightEigenvectors \widetilde{\orthMatrix}\|
      \\
      &\le&
      \left(\frac{2\sqrt{2K}(\|\bX\| + \|\bPi\|)^2}{\lambda_{\ntopics}^2( \bPi)} +1\right)\|\bX - \bPi\|,
    \end{EQA}
    where the last inequality is due to Corollary~\ref{lemma:eigenvectorConsistency}.
    Next,
    \begin{EQA}
      \|\adjacencyEigenvalues^{-1} - \widetilde{\orthMatrix}^{\T} \probEigenvalues^{-1} \orthMatrix\| 
      &=&
      \|\adjacencyEigenvalues^{-1}(\orthMatrix^{\T} \probEigenvalues \widetilde{\orthMatrix} - \adjacencyEigenvalues) \widetilde{\orthMatrix}^{\T} \probEigenvalues^{-1} \orthMatrix\|
      \\
      &\le&
      \|\adjacencyEigenvalues^{-1}\| \ \|\adjacencyEigenvalues - \orthMatrix^{\T} \probEigenvalues \widetilde{\orthMatrix}\| 
      \ \|\probEigenvalues^{-1}\|
      \\
      &\le&
  \left(\frac{2\sqrt{2K}(\|\bX\| + \|\bPi\|)^2}{\lambda_{\ntopics}^2( \bPi)} +1\right)\frac{\|\bX - \bPi\|}{\lambda_{\ntopics}(\bX)  \lambda_{\ntopics}(\bPi)},
    \end{EQA}
    where \(\lambda_{\ntopics}(\bX)\) is the \(\ntopics\)-th largest singular value of matrix \(\bX\). 
    Due to Weyl's inequality and the fact that \(\|\bX - \bPi\| \le \frac{1}{2} \lambda_{\ntopics}(\bPi)\) we have $\lambda_{\ntopics}(\bX)\ge \frac{1}{2}\lambda_{\ntopics}(\bPi)$ and $\|\bX\| \le \|\bPi\| + \frac{1}{2} \lambda_{\ntopics}(\bPi) \le \frac{3}{2}\|\bPi\|$. Plugging these inequalities in the last two displays we obtain the lemma. 
  \end{proof}

\subsection{Noisy Separable Matrix Factorization}
\label{sec:consistencySPA}
  In this section, we give a bound on the error of preconditioned SPA in Noisy Separable Matrix Factorization model. Assume that we observe 
  \begin{EQA}[c]
    \noisyTargetMatrix = \targetMatrix + \noiseMatrix = {\bW} {\bf Q} + \noiseMatrix,
  \end{EQA}
  where \(\noiseMatrix \in \RR^{\nsize \t \ntopics}\) is a perturbation (noise) matrix, and
  \begin{EQA}[c]
    \targetMatrix = {\bW} {\bf Q} , 
  \end{EQA}
  where \({\bW} \in \RR_{+}^{\nsize \t \ntopics}\) and ${\bf Q}  \in \RR^{\ntopics \t \ntopics}$.  If we assume that ${\bW}$ satisfies Assumption~\ref{ass:anchorDoc} then we obtain the setting usually referred to as Noisy Separable Matrix Factorization (NSMF).

  The following theorem holds for preconditioned SPA in the NSMF model, see~\cite{Gillis2015,Mizutani2016}. 

 { \begin{theorem}
  \label{theorem:spaBasic}
    Let \(\ntopics \geq 2\) and let  Assumption~\ref{ass:anchorDoc} hold. Assume that matrix ${\bf Q}$ is non-degenerate and the entries $W_{im}$ of matrix $\bW$ satisfy the condition $\sum_{m=1}^K W_{im}\le 1$ for $i=1,\dots,n$. Moreover, assume that for any $i=1,\dots,n$, the norms of the rows 
    of matrix \(\noiseMatrix\) satisfy $\|{\bf e}_{i}^{\rm T}\noiseMatrix\|_{2}\leq \epsilon$ with  
    \begin{EQA}[c]
    \epsilon   \leq C_* \frac{\lambda_{\min}({\bf Q})}{K \sqrt{K}}
    \end{EQA}
    for some constant $C_* > 0$ small enough. Let $J$ be the set of indices returned by the preconditioned SPA with input \((\noisyTargetMatrix, \ntopics)\). Then, there exist a constant \(C_0 > 0\) and a permutation \(\pi\) such that, for all $j \in J$,
    \begin{EQA}[c]
      \|\tilde{\gv}_{j} - {\bf q}_{\pi(j)}\|_2 \le C_0 \condNumber({\bf Q}) \epsilon,
    \end{EQA}
  where \(\tilde{\gv}_{k}\) and \({\bf q}_{k}\) are the k-th rows of matrices \(\noisyTargetMatrix\) and ${\bf Q}$, respectively.
  \end{theorem}}

  Note that this error bound depends on the upper bound on the individual errors \(\|{\bf e}_{i}^{\rm T}\noiseMatrix\|_2\). From the statistical point of view, one might expect that there should be an algorithm, which improves upon this error bound if there are many nearly ``pure'' rows in matrix \(\targetMatrix\), so that the value of the error is diminished by averaging. However, to the best of our knowledge, no such algorithm complemented with a performance analysis can be found in the literature.

  We now consider a specific instance of NSMF model given by~\eqref{eq:umodel}. In this case, \(\noisyTargetMatrix = \empiricalLeftEigenvectors\) and ${\bf Q}={\bf H}{\bf O}$  for an orthogonal matrix ${\bf O}$. Specifically, ${\bf O}$ is the orthogonal matrix, for which~\eqref{DKSingularVectors} holds (it is the same matrix ${\bf O}$, for which the bound of Lemma~\ref{lemma:rowFactorBound} is valid).  Combining Theorem~\ref{theorem:spaBasic} with Lemma~\ref{lemma:rowFactorBound} 
  we get the following corollary. 
  \begin{corollary}
  \label{corollary:consistencyBasis}
    Let Assumptions~\ref{ass:anchorDoc}  and~\ref{cond:community memberships} be satisfied with constant $\bar C\le C_*$. Consider the matrices $\bPi$, $\bX$, ${\bf H}$, $\empiricalLeftEigenvectors$ as in~\eqref{svc:Pi} -- \eqref{svc:X} such that $\lambda_{\ntopics}(\bPi)>0$ and {\(\|\bX - \bPi\| \le \frac{1}{2} \lambda_{\ntopics}(\bPi)\)}. Let ${\bf O}$ be the orthogonal matrix, for which \eqref{DKSingularVectors} holds.  
     Let $J$ be the set of indices returned by the preconditioned SPA with input \((\empiricalLeftEigenvectors, \ntopics)\), and let \(\basisMatrixEstimate = \empiricalLeftEigenvectors_J\). 
     Then, there exist a constant \(C_0>0\) and a permutation \(\pi\) such that, for all \(j = 1, \dots, \ntopics\),
    \begin{equation}
    \label{eq:consistencyBasis}
      \bigl\|\hat{\hv}_{j} - {\bf q}_{\pi(j)}\bigr\|_2 \le C_0 \condNumber(\basisMatrix) \errorAdjacency,
    \end{equation}
      where \(\hat{\hv}_{k}\) and \({\bf q}_{k}\) are the k-th rows of matrices $\hat{\bf H}$ and ${\bf H}{\bf O}$, respectively.
    Furthermore,
    \begin{equation}
      \label{eq:factorBound}
      \bigl\|\basisMatrixEstimate - \tilde\bP \basisMatrix \mathbf{O}\bigr\|_{F} \le C_0 \ntopics^{1/2} \condNumber(\bW) \errorAdjacency,
    \end{equation}
    where \(\tilde\bP\) is a permutation matrix corresponding to the permutation \(\pi\). 
  \end{corollary}
  \begin{proof}
    Taking into account equations~\eqref{svc:Pi} -- \eqref{eq:umodel} we apply Theorem~\ref{theorem:spaBasic} with ${\bf Q}={\bf H}{\bf O}$, ${\bf N}=\empiricalLeftEigenvectors- {\bf U}{\bf O}$. By Lemma~\ref{lemma:rowFactorBound},
    $$
    \|{\bf e}_{i}^{\rm T}\noiseMatrix\|_2 = \|{\bf e}_{i}^{\rm T}(\empiricalLeftEigenvectors- {\bf U}{\bf O})\|_2 \le \epsilon,\quad i=1,\dots,n,
    $$
    where $\epsilon =\errorAdjacency$. Therefore, using Assumption~\ref{assumption_eigenvalues}, \eqref{eq:matrix_property6}, and the fact that $\bar C\le C^*$ we have
    $$
    \epsilon\le  \frac{\bar C}{\lambda_{1}(\bW) K \sqrt{K}}
    = \frac{\bar C \lambda_{\min}({\bf H})}{ K \sqrt{K}}
    \le  \frac{C_*\lambda_{\min}({\bf H}{\bf O})}{K \sqrt{K}}.
    $$
    Thus, the assumptions of Theorem~\ref{theorem:spaBasic} are satisfied and we deduce from Theorem~\ref{theorem:spaBasic} that  
    $$
    \bigl\|\hat{\hv}_{j} - {\bf q}_{\pi(j)}\bigr\|_2 \le C_0 \condNumber({\bf HO})\epsilon = C_0\condNumber({\bf H}) \errorAdjacency,
    $$
    where \(\hat{\hv}_{k}\) and \({\bf q}_{k}\) are the $k$-th rows of matrices $\hat{\bf H}$ and ${\bf H}{\bf O}$, respectively. Thus, \eqref{eq:consistencyBasis} follows. Inequality~\eqref{eq:factorBound} is an immediate consequence of~\eqref{eq:consistencyBasis} and of the equality $\kappa({\bf H})=\kappa(\bW)$ (cf. \eqref{eq:matrix_property4}).
  \end{proof}


\subsection{Concentration Bounds for Multinomial Matrices}
   In this section, we provide a bound with high probability on the spectral norm of matrix $\bX-\bPi$. Recall that, by definition, \(\bX^\T=[X_1, \dots, X_{n}]\) is such that \(N X_i \in \mathbb{R}^{p}\) are independent random vectors distributed according to \(p\)-dimensional multinomial distribution with parameters \((N, \Pi_i)\). We will use matrix Bernstein inequality in the following form (cf. Theorem 6.1.1 in~\cite{tropp2015introduction}).
  \begin{proposition}[Matrix Bernstein inequality]
    \label{prop:MatrixBernstein}
      Let \(\bZ_1,\dots ,\bZ_N\) be independent zero-mean \(n\times p\) random matrices such  that  \(\|\bZ_m\| \le L\) for $m=1,\dots,N$. Then, for all \(t>0\) we have
      \begin{EQA}[c]
        \P \left(\left\Vert \dfrac{1}{N}\sum_{m=1}^{N} \bZ_m\right\Vert\ge t\right)\leq (n+p)\exp \left(-\dfrac{t^2N^2}{2(\sigma^2 + LtN/3)}\right),
      \end{EQA} 
      where 
      \begin{equation*}
      \sigma^2 = \max\left \{\left \Vert \sum_{m = 1}^{N} \E\left (\bZ_m\bZ^{\T}_m\right )\right \Vert, \left \Vert \sum_{m=1}^{N} \E\left (\bZ_m^{^\T}\bZ_m\right )\right \Vert\right \}.
    \end{equation*}
    \end{proposition}
  Applying Proposition~\ref{prop:MatrixBernstein} to our setting we obtain the following result. 
  \begin{proposition}
    \label{prop:MatrixBernstein-multinomial}
    Let \(\bX^\T=[X_1, \dots, X_{n}]\) be such that \(N X_i \in \mathbb{R}^{p}\) are independent random vectors distributed according to \(p\)-dimensional multinomial distribution with parameters \((N, \Pi_i)\). Then, for all \(t>0\) we have
    \begin{equation}
    \label{eq:prop:MatrixBernstein-multinomial}
      \P \left(\left\Vert \bX-\bPi\right\Vert\ge t\right)\leq (n+p)\exp \left(-\dfrac{t^2N}{2\sqrt{2}(n + t\sqrt{n}/3)}\right).
    \end{equation} 
  \end{proposition}
  \begin{proof}
    We prove~\eqref{eq:prop:MatrixBernstein-multinomial} for $\bX^\T-\bPi^\T$ rather than $\bX-\bPi$, which is equivalent.
    Matrix $\bZ^\T=\bX^\T-\bPi^\T$ has the form $\bZ^\T = [Z_1, \dots, Z_{n}]$ with independent column vectors  $Z_i = \frac{1}{N} \sum_{m = 1}^{N} (T_{im}-\E(T_{im}))$, where vectors \(T_{im}\) are distributed according to  \(p\)-dimensional multinomial distribution with parameters \((1, \Pi_i)\) and independent over $m$ for any fixed $i$. Here, we have used the fact that Multinomial$_p$\((N, \Pi_i)\) is a sum of $N$ independent Multinomial$_p$\((1, \Pi_i)\) vectors. 
    We also have $\bPi^\T=[\Pi_1, \dots, \Pi_{n}]$.
    Thus, we can write
    \begin{equation}\label{lemma_spectral1}
      \bZ^\T=\frac{1}{N}\sum_{m=1}^{N} \left(\bT_{m} - \E(\bT_{m})\right) = \frac{1}{N} \sum_{m = 1}^{N}\bZ_{m},
    \end{equation}
    where \(\bT_{m}=[T_{1m},\dots, T_{nm}]\) and \(\bZ_{m}=\bT_{m}-\E(\bT_{m})\) are independent zero-mean random matrices. 
    
    We apply Proposition~\ref{prop:MatrixBernstein} to the sum~\eqref{lemma_spectral1}. The first step is to evaluate \(\left \Vert \sum_{m=1}^{N} \E\left (\bZ_m\bZ^{\T}_m\right )\right \Vert\). Let $T_{im}(k)$ denote the $k$-th component of $T_{im}$, $k=1,\dots,p$. We have $\E(T_{im}(k)) = \Pi_{ik}$, ${\rm Var}(T_{im}(k)) = \Pi_{ik}(1-\Pi_{ik})$,  ${\rm Cov}(T_{im}(k), T_{im}(j)) = -\Pi_{ik}\Pi_{ij}$ for $i\ne j$. Therefore,
    \begin{align*}
      \E\left (\bZ_m\bZ^{\T}_m\right )&=\E\left (\bT_{m}\bT^{\T}_{m}\right )-\E\left (\bT_{m}\right )\E\left (\bT^{\T}_{m}\right)
      \\
      &=\E\sum_{i=1}^{n} T_{im}T_{im}^\T - \bPi\bPi^\T=\sum_{i=1}^{n} \bY_{i},
    \end{align*}
    where 
    \begin{equation*}
      \bY_{i}
      = {\rm diag}(\Pi_{i1},\dots,\Pi_{ip}) - \Pi_i\Pi_i^\T.
    \end{equation*}
    The spectral norm of  $\bY_{i}$ satisfies
    \begin{align*}
      \left \Vert \bY_{i}\right \Vert^{2}\leq \left \Vert \bY_{i}\right \Vert^{2}_{F}=\sum_{k=1}^{p} \Pi_{ik}^{2}+\Big(\sum_{k=1}^{p} \Pi_{ik}^{2} \Big)^{2}-2\sum_{k=1}^{p} \Pi_{ik}^{3}\leq 2,
    \end{align*}
    where we have used the fact that \(\sum_{k=1}^{p} \Pi_{ik}^{2}\le \sum_{k=1}^{p} \Pi_{ik}=1\).
    Thus,  \(\left \Vert\E\left (\bZ_m\bZ^{\T}_m\right )\right \Vert \leq \sqrt{2}n\) and 
    \begin{equation}\label{lemma_spectral2}
      \left \Vert \sum_{m=1}^{N} \E\left (\bZ_m\bZ^{\T}_m\right )\right \Vert\leq \sqrt{2} N n.
    \end{equation}
    Next, we derive an upper bound on 
    \(\left \Vert \sum_{m=1}^{N} \E\left (\bZ^\T_m\bZ_m\right )\right \Vert\). 
    Note that $\E (\bT^{\T}_{m}\bT_{m} )$ is a matrix with diagonal entries $\E (T_{im}^{\T} T_{im} )= \sum_{k=1}^{p} \Pi_{ik}=1$ while its off-diagonal entries are $\E (T^{\T}_{im} T_{jm} )=[\E (T_{im})]^{\T}\E (T_{jm})= \Pi_{i}^\T \Pi_j$ due to independence between $T_{im}$ and $T_{jm}$ for $i\ne j$.  Also, $\E(\bT^{\T}_{m})\E(\bT_{m})= \bPi^\T \bPi$ is a matrix with entries $ \Pi_{i}^\T \Pi_j$. Hence,
    \begin{align}\label{lemma_spectral3}
      \E\left (\bZ^{\T}_m\bZ_m\right )=\E(\bT^{\T}_{m}\bT_{m})-\E(\bT^{\T}_{m})\E (\bT_{m})
      = {\rm diag}\Big(1-\Vert \Pi_1 \Vert_2^2,\dots,1-\Vert \Pi_n \Vert_2^2\Big).
    \end{align}
    It follows that 
    \(\left \Vert \E\left (\bZ^{\T}_m\bZ_m\right )\right \Vert \leq 1\), and thus \(\left \Vert \sum_{m=1}^{N} \E\left (\bZ^\T_m\bZ_m\right )\right \Vert\le N\).  Combining this inequality with~\eqref{lemma_spectral2} we obtain that $\sigma^2$ defined in Proposition~\ref{prop:MatrixBernstein} satisfies \(\sigma^2\leq\sqrt{2}N n\). 

    Finally, we specify the constant $L$ that gives an upper bound on \(\left \Vert \bZ_{m} \right \Vert\). Let \(u \in S^{p-1}\) be an element of the unit sphere in \(\mathbb R^{p}\). Since for any \(i\) vector $T_{im}$ has only one component equal to 1 and all other components 0  we have $\Vert T_{im}-\E(T_{im})\Vert_2^2 =\Vert T_{im}-\Pi_{i}\Vert_2^2 \le 2$ and thus
    \begin{EQA}[c]
      \left \vert u^{\T}(T_{im}-\E(T_{im}))\right \vert \leq \sqrt{2}.
    \end{EQA}
    It follows that 
    \begin{align*}
      \left \Vert \bT_{m}-\E (\bT_{m}) \right \Vert^{2} &= \underset{u\in S^{p-1}}{\sup}\left \Vert u^{\T}(\bT_{m}-\E(\bT_{m}))\right \Vert^{2}_{2}
      \\
      &=
      \underset{u \in S^{p-1}}{\sup} \sum_{i=1}^{n}\left \vert u^{\T}(T_{im}-\E(T_{im}))\right \vert^{2}\leq 2n
    \end{align*}
    and we get \(\left \Vert \bZ_{m} \right \Vert\leq \sqrt{2n} = : L\) for any \(m = 1, \dots, n\).
    The desired result now follows by applying 
     Proposition~\ref{prop:MatrixBernstein} with $\sigma^2\leq\sqrt{2}N n$ and $L=\sqrt{2n}$.  
  \end{proof}

  The next lemma is a corollary of Proposition~\ref{prop:MatrixBernstein-multinomial}:
  %
  \begin{lemma}
  \label{lemma:spectralNormNoise}
    Let the assumptions of Proposition~\ref{prop:MatrixBernstein-multinomial} be satisfied.
    Assume that  $N\ge \log(n+p)$ and $\min(n,p)\ge 2$. 
    Then
    \begin{equation}\label{eq1:lemma:spectralNormNoise}
      \P \left(\left\Vert \bX-\bPi\right\Vert\ge 4\sqrt{\frac{n\log(n+p)}{N}}\right)\leq (n+p)^{-1}.
    \end{equation}
    Furthermore,
    \begin{equation}\label{eq2:lemma:spectralNormNoise}
      \P \left(\max_{1\le i\le n} \|\ev_{i}^{\T}(\bX - \bPi) \|_{2} 
      \ge 5\sqrt{\frac{\log(n+p)}{N}}\right)\leq (n+p)^{-1}.
    \end{equation}
  \end{lemma}
  \begin{proof}
    Inequality~\eqref{eq1:lemma:spectralNormNoise} follows easily from Proposition~\ref{prop:MatrixBernstein-multinomial} by setting $t=4\sqrt{\frac{n\log(n+p)}{N}}$ and using the assumptions $N\ge \log(n+p)$.
    In order to prove~\eqref{eq2:lemma:spectralNormNoise}, we bound each probability $\P \left( \|\ev_{i}^{\T}(\bX - \bPi) \|_{2} 
    \ge 5\sqrt{\frac{\log(n+p)}{N}}\right)$
    via Proposition~\ref{prop:MatrixBernstein-multinomial} with $n=1$ (that is, we apply Proposition~\ref{prop:MatrixBernstein-multinomial} to $1 \times p$ matrices $\ev_{i}^{\T}\bX$, $\ev_{i}^{\T}\bPi$) and then use the union bound. This yields
    \begin{align*}
      \P \left(\max_{1\le i\le n} \|\ev_{i}^{\T}(\bX - \bPi) \|_{2} 
      \ge 5\sqrt{\frac{\log(n+p)}{N}}\right)\leq n(p+1)\exp \left(-\dfrac{75\log(n+p)}{16\sqrt{2}}\right).
    \end{align*}
    The right hand side of this inequality does not exceed $(n+p)^{-1}$.
  \end{proof}


\section{Proofs of the Main Results}
\label{sec:proofs}
\subsection{Proof of Lemma~\ref{lemma:rowFactorBound}}
   Using the fact that $\widehat{\probRightEigenvectors}_1^{\T}\widehat{\probRightEigenvectors}^{\T}=0$ we obtain
  \begin{EQA}
    && \|\ev_{i}^{\T}(\empiricalLeftEigenvectors - \probLeftEigenvectors \orthMatrix)\|_2
    =
    \|\ev_{i}^{\T}(\bX \empiricalRightEigenvectors \adjacencyEigenvalues^{-1} - \bPi \probRightEigenvectors  \probEigenvalues^{-1} \orthMatrix)\|_2
    \\
    &=&
    \|\ev_{i}^{\T} \bX \empiricalRightEigenvectors (\adjacencyEigenvalues^{-1} - \widetilde{\orthMatrix}^{\T} \probEigenvalues^{-1} \orthMatrix) + \ev_{i}^{\T} \bX (\empiricalRightEigenvectors - \probRightEigenvectors \widetilde{\orthMatrix}) \widetilde{\orthMatrix}^{\T} \probEigenvalues^{-1} \orthMatrix + \ev_{i}^{\T}(\bX - \bPi) \probRightEigenvectors \probEigenvalues^{-1} \orthMatrix\|_2
    \\
    &\le&
    \|\ev_{i}^{\T} \bX \empiricalRightEigenvectors (\adjacencyEigenvalues^{-1} - \widetilde{\orthMatrix}^{\T} \probEigenvalues^{-1} \orthMatrix)\|_2 + \|\ev_{i}^{\T} \bX (\empiricalRightEigenvectors - \probRightEigenvectors \widetilde{\orthMatrix}) \widetilde{\orthMatrix}^{\T} \probEigenvalues^{-1} \orthMatrix\|_2 + \|\ev_{i}^{\T}(\bX - \bPi) \probRightEigenvectors \probEigenvalues^{-1} \orthMatrix\|_2
    \\
    &=&
    G_1 + G_2 + G_3.
  \end{EQA}
  We now bound the values $ G_1, G_2$ and $G_3$ separately. We have
  \begin{EQA}
    G_1
    &=&
    \|\ev_{i}^{\T} \bX \empiricalRightEigenvectors (\adjacencyEigenvalues^{-1} - \widetilde{\orthMatrix}^{\T} \probEigenvalues^{-1} \orthMatrix)\|_2
    \le
    \|\ev_{i}^{\T} \bX\|_2 \, \|\empiricalRightEigenvectors\| \, \|\adjacencyEigenvalues^{-1} - \widetilde{\orthMatrix}^{\T} \probEigenvalues^{-1} \orthMatrix\|
   \\
   &\le&
    C \ntopics^{1/2} \condNumber^2(\bPi) ~ \frac{\|\ev_{i}^{\T} \bX\|_2\, \|\bX - \bPi\|}{\lambda_{\ntopics}^{2}(\bPi)},
  \end{EQA}
  where the last inequality is due to Lemma~\ref{corollary:eigenvalues}. 
  The values $G_2$ and $G_3$ can be controlled using the bounds for the norm of matrix product and Corollary~\ref{lemma:eigenvectorConsistency}:
  \begin{EQA}
     G_2
    &=&
    \|\ev_{i}^{\T} \bX (\empiricalRightEigenvectors - \probRightEigenvectors \widetilde{\orthMatrix}) \widetilde{\orthMatrix}^{\T} \probEigenvalues^{-1} \orthMatrix\|_2
    \le
    \|\ev_{i}^{\T} \bX\|_2 \, \|\empiricalRightEigenvectors - \probRightEigenvectors \widetilde{\orthMatrix}\| \, \|\probEigenvalues^{-1}\|
     \\
   & =&
    \frac{\|\ev_{i}^{\T} \bX\|_2 \, \|\empiricalRightEigenvectors - \probRightEigenvectors \widetilde{\orthMatrix}\|}{\lambda_{\ntopics}(\bPi)}
    \le
    5\sqrt{2}\kappa(\bPi) \frac{\|\ev_{i}^{\T} \bX\|_2 \, \|\bX - \bPi\|}{\lambda_{\ntopics}^2(\bPi)}
  \end{EQA}
  \begin{EQA}
    G_3
    &=&
    \|\ev_{i}^{\T}(\bX - \bPi) \probRightEigenvectors \probEigenvalues^{-1} \orthMatrix\|_2
    \le
    \|\ev_{i}^{\T}(\bX - \bPi) \|_2 \, \|\probRightEigenvectors\| \, \|\probEigenvalues^{-1}\|
    =
    \frac{\|\ev_{i}^{\T}(\bX - \bPi)\|_2}{\lambda_{\ntopics}(\bPi)}.
  \end{EQA}
  Combining these bounds proves the lemma.

\subsection{Proof of Lemma~\ref{lemma:documenTopicMatrixBound}}
  We first prove that matrix $\basisMatrixEstimate$ is non-degenerate.
  In this proof, we denote by ${\bf O}$ the orthogonal matrix, for which \eqref{DKSingularVectors} holds, and by $\tilde\bP$ the permutation matrix, for which the bound of Corollary~\ref{corollary:consistencyBasis} holds. Using Weyl's inequality~\cite[Theorem C.6]{Giraud-book} and Corollary~\ref{corollary:consistencyBasis} we obtain
  \begin{align*}
    \lambda_{\min}(\basisMatrixEstimate) &\geq \lambda_{\min}(\tilde\bP\basisMatrix \orthMatrix) - \|\basisMatrixEstimate - \tilde\bP\basisMatrix \orthMatrix \| \\
    & \geq \lambda_{\min}(\basisMatrix) - C_0 K^{1/2}\kappa(\bW)\beta(\bX, \bPi).
  \end{align*}
  Using this inequality, Assumption~\ref{cond:community memberships} with $\bar C\le C_0^{-1}$, and equations~\eqref{eq:matrix_property6}, \eqref{eq:matrix_property4} we find
  \begin{align}\label{eq:lh}
    \lambda_{\min}(\basisMatrixEstimate) &\geq \lambda_{\min}(\basisMatrix) - \frac{1}{2\lambda_{1}(\bW)}= \frac{1}{2\lambda_{1}(\bW)},
  \end{align}
  which proves that $\basisMatrixEstimate$ is invertible. Then, for the estimator
  $
    \documentTopicMatrixEstimate = \empiricalLeftEigenvectors \basisMatrixEstimate^{-1}
  $ and the permutation matrix $\bP=\tilde\bP^{-1}$ we have
  \begin{align*}
    \|\documentTopicMatrixEstimate - \documentTopicMatrix \bP\|_F &= \|\empiricalLeftEigenvectors\basisMatrixEstimate^{-1} - \probLeftEigenvectors\basisMatrix^{-1}\bP\|_F\\
    &\leq \|\empiricalLeftEigenvectors(\basisMatrixEstimate^{-1} - \orthMatrix^{\T} \basisMatrix^{-1}\bP)\|_F + \|(\empiricalLeftEigenvectors - \probLeftEigenvectors\orthMatrix)[\bP^{-1} \basisMatrix\orthMatrix]^{-1}\|_F \\
    &= 
    I_1 + I_2.
  \end{align*}
  We now bound separately $I_1$ and $I_2$. Due to~\eqref{eq:lh} and~\eqref{eq:matrix_property6} we have
  \begin{align}\label{eq:lh1}
        \|\basisMatrixEstimate^{-1}\|&= \frac{1}{\lambda_{\min}(\basisMatrixEstimate)}\le 2\lambda_{1}(\bW), \quad   \|\basisMatrix^{-1}\|=  \lambda_{1}(\bW).   
  \end{align}
   Using the fact that $\|{\bf A}^{-1}-{\bf B}^{-1} \|_F\le \|{\bf A}^{-1} \| \,\|{\bf B}^{-1} \|\,\|{\bf A}-{\bf B} \|_F$ with ${\bf A}=\basisMatrixEstimate$, ${\bf B}=\bP^{-1} \basisMatrix \orthMatrix=\tilde\bP \basisMatrix \orthMatrix$, inequality~\eqref{eq:lh1} and Corollary~\ref{corollary:consistencyBasis} we find
  \begin{align*}
    I_1
    =
    \|\empiricalLeftEigenvectors(\basisMatrixEstimate^{-1} - \orthMatrix^{\T}\basisMatrix^{-1}\bP)\|_F &\leq \|\basisMatrixEstimate^{-1} - \orthMatrix^{\T}\basisMatrix^{-1}\bP\|_F
    \\
    &\le \|\basisMatrixEstimate^{-1}\| \, \|\orthMatrix^{\T}\basisMatrix^{-1}\bP\| \, \|\basisMatrixEstimate - \tilde\bP \basisMatrix \orthMatrix\|_{F} \\
    &\leq 2 C_0 \ntopics^{1/2} \lambda_{1}^{2}(\documentTopicMatrix) \kappa( \bf \bW ) \errorAdjacency.
  \end{align*}
  On the other hand,
  \begin{align*}
    I_2 &= \|(\empiricalLeftEigenvectors - \probLeftEigenvectors\orthMatrix)[\bP^{-1} \basisMatrix\orthMatrix]^{-1}\|_F
    \leq \|\empiricalLeftEigenvectors - \probLeftEigenvectors\orthMatrix\|_F \, \|\basisMatrix^{-1}\|\\
    & \le 
    \frac{5\sqrt{2K}\kappa(\bPi)\|\bX - \bPi\|}{\lambda_{\ntopics}( \bPi)} \|\basisMatrix^{-1}\|, 
  \end{align*}
  where the last inequality follows from Corollary~\ref{lemma:eigenvectorConsistency}.
  Combining the above bounds we get 
  \begin{EQA}[c]
    \|\documentTopicMatrixEstimate - \documentTopicMatrix \bP\|_F \leq C \ntopics^{1/2}\lambda_{1}(\documentTopicMatrix) \left \{ \lambda_{1}(\documentTopicMatrix) \kappa( \bf \bW) \errorAdjacency +  
  \frac{\kappa(\bPi)\|\bX - \bPi\|}{\lambda_{\it K}( \bPi)}
    \right \}.
  \end{EQA}
  %

\subsection{Proof of Theorem~\ref{theorem:mainBound}}
\label{Appendix:proof-of-thm1}
  We apply Lemma~\ref{lemma:documenTopicMatrixBound} combined with the concentration inequalities of Lemma~\ref{lemma:spectralNormNoise}. 
  First, we check that Assumption~\ref{cond:community memberships} holds with probability at least $1 - 2 (n + p)^{-1}$. For $N\geq \log(n+p)$ we get from Lemma~\ref{lemma:spectralNormNoise} that  
  \begin{EQA}[c]
  \|\bX - \bPi\| \le 4 \sqrt{\dfrac{n\log(n+p)}{N}}
  \end{EQA}
  with probability at least $1 - 1 / (n + p)$, and   
  \begin{EQA}[c]
    \max_{1\le i\le n}\|\ev_{i}^{\T}(\bX - \bPi) \|_{2}  \le  5\sqrt{\frac{\log(\nsize + p)}{N}}
  \end{EQA}
  with probability at least $1-1/(n+p)$. 
   Notice also that $\max_{i} \|\ev_{i}^{\T} \bX\|_{2} = \max_{i} \sqrt{\sum_{j=1}^p X_{ij}^2} \le 
   \max_{i} \sqrt{\sum_{j=1}^p X_{ij}} = 1$. 
 Putting together the above remarks we deduce that, with probability at least \(1 - 2(\nsize+p)^{-1}\),
  \begin{EQA}\label{upper_bound_beta}
    \errorAdjacency
    &\le&
    5 \left\{\condNumber^2(\bPi) \sqrt{\ntopics n} + {\lambda_{\ntopics}(\bPi)}\right \}
    \dfrac{\sqrt{\log(n +p)}}{\lambda_{\ntopics}^2(\bPi)\sqrt{N}}
    \\
      &\leq&
      10 \condNumber^2(\bPi) \sqrt{\ntopics n} 
      \dfrac{\sqrt{\log(n +p)}}{\lambda_{\ntopics}^2(\bPi)\sqrt{N}}
  \end{EQA}
   where we have used the inequality  $\lambda_K(\bPi) \leq \sqrt{n/K}$ 
   proved in Lemma~\ref{lemma:matrix_eigen}. Since $\lambda_K(\bPi)$ is chosen to satisfy~\eqref{cond:lambda}
   we get that, with probability at least \(1 - 2(\nsize+p)^{-1}\),
    \begin{align*}
      \errorAdjacency &\leq  \frac{\bar C}{\lambda_{1}(\bW)\kappa(\bW)K\sqrt{K}}.
    \end{align*}
    Thus, on an event that has probability at least \(1 - 2(\nsize+p)^{-1}\), Assumption~\ref{cond:community memberships} is satisfied and we can apply Lemma~\ref{lemma:documenTopicMatrixBound}. This yields that, with probability at least \(1 - 2(\nsize+p)^{-1}\),
  \begin{align*}
    \min_{\bP\in \mathcal{P}}\bigl\|\documentTopicMatrixEstimate - \documentTopicMatrix \bP\bigr\|_{F}
    & \le
    C \ntopics^{1/2} \lambda_{1}(\documentTopicMatrix)\left \{ \lambda_{1}(\documentTopicMatrix) \kappa(\bW) \errorAdjacency + \frac{\kappa(\bPi)\|\bX - \bPi\|}{\lambda_{\ntopics}(\bPi)}\right \}
    \\
    &\le
    C \dfrac{\lambda_{1}(\documentTopicMatrix)\sqrt{n\ntopics \log(n+p)}}{\sqrt{N}\lambda_{\ntopics}(\bPi)} \left \{ \lambda_{1}(\documentTopicMatrix) \kappa(\bW) \dfrac{\ntopics^{1/2} \condNumber^2(\bPi)}{\lambda_{\ntopics}(\bPi)} 
    + \kappa(\bPi)\right \}
    \\
    &\le
    C \ntopics \sqrt{\dfrac{n \log(n + p)}{N}}\left(\dfrac{\lambda_{1}(\documentTopicMatrix)}{\lambda_{\ntopics}(\bPi)}\right)^2 \kappa(\bW) \kappa^2(\bPi),
  \end{align*}
  where we have used the inequalities   $\lambda_K(\bPi) \leq \sqrt{n / K}$ and $\lambda_{1}(\bPi) \leq \sqrt{K} \lambda_{1}(\bW)$ (see Lemma~\ref{lemma:matrix_eigen}).

\subsection{Proof of Theorem~\ref{theorem:mainBound-adapt} and Corollary~\ref{corollary:mainBound-adapt}}
 We will use the following lemma.
 \begin{lemma}\label{lem:thresholding}
 Let $\bPi \in \RR^{\nsize \t p}$ be a rank \(\ntopics\)
    matrix with smallest non-zero singular value \(\lambda_{\ntopics}(\bPi)\), and
    $\bX \in \RR^{\nsize \t p}$ be a  matrix such that
    \(\|\bX - \bPi\| \le \tau\) for some $\tau>0$. Let
    $
 \hat K = \max\{j: \ \lambda_j(\bX)>\tau \}.
 $
 If \(\lambda_{\ntopics}(\bPi)> 2\tau\) then  $\hat K =K$.
 \end{lemma}
 \begin{proof}
   By Weyl's inequality, we have $|\lambda_j(\bX)-\lambda_j(\bPi)|\le \tau$ for all $j$. Since $\lambda_j(\bPi)=0$ for $j\ge K+1$ we deduce that $\hat K \le K$.
   On the other hand, $\hat K \ge K$. Indeed, condition \(\lambda_{\ntopics}(\bPi)> 2\tau\) implies that $\lambda_K(\bX)\ge \lambda_K(\bPi)- |\lambda_j(\bX)-\lambda_j(\bPi)|>\tau$.
 \end{proof}
 
 Theorem~\ref{theorem:mainBound-adapt} is obtained by combining Theorem~\ref{theorem:mainBound} with Lemma \ref{lem:thresholding}.
 Indeed, notice that the bound of Theorem~\ref{theorem:mainBound} is proved on the event $\mathcal{A}:=\Big\{\left\Vert \bX-\bPi\right\Vert\le 4\sqrt{\frac{n\log(n+p)}{N}}\Big\}$.  
Set $\tau=  4\sqrt{\frac{n\log(n+p)}{N}}$. It follows from Lemma \ref{lem:thresholding} that if 
\begin{equation}\label{condition:thresh}
  \lambda_K(\bPi)> 8\sqrt{\frac{n\log(n+p)}{N}}
\end{equation}
then on the event $\mathcal{A}$ we have $\hat K =K$. But condition \eqref{condition:thresh} is implied by \eqref{cond:lambda} and \eqref{cond:lambda-adapt}. Therefore, the proof of Theorem~\ref{theorem:mainBound} goes through verbatim if we replace $K$ by $\hat K$. This yields Theorem~\ref{theorem:mainBound-adapt}.
 Corollary~\ref{corollary:mainBound-adapt} is deduced from Theorem~\ref{theorem:mainBound-adapt} in the same way as Corollary~\ref{corollary:mainBound} was deduced from Theorem~\ref{theorem:mainBound}. 
 

\subsection{Proof of Theorem~\ref{theorem:lower_bound}}
\label{sec:proof-of-lower-bound}
  
We use the techniques of proving minimax lower bounds based on a reduction to the problem of testing multiple hypotheses \cite[Chapter 2]{tsybakov2008introduction}. The hypotheses correspond to probability measures $\mathbb{P}_{\bPi^{(j)}}$, where $\bPi^{(j)}=\bW^{(j)}\bA$ with carefully chosen matrix $\bA$ and matrices $\bW^{(j)}$, $j=0,1,\dots,T$. The construction of these matrices borrows some elements from the proofs of the lower bounds in \cite{Ke2017,bing2018fast}. An additional subtlety is related to the fact that we need to grant Assumption \ref{assumption_eigenvalues} on the singular values. 
Without loss of generality we assume that $n$ is a multiple of $K$ and that $K$ is even. 

\smallskip

  1. {\it Construction of the set of matrices $\bW^{(j)}$.}\\
  We first introduce the basic matrix $\bW^{(0)}$ and then define matrices $\bW^{(j)}$, $j=1,\dots,T$ as slightly perturbed versions of $\bW^{(0)}$.
  
  Let $\bD_1$ be a $n \times K$ matrix composed of $n/K$ blocks, each of which is the identity matrix $\bI_K$ of size $K$:
  \begin{align*}
    \bD_1^\T 
    &= \left[ 
    \begin{array}{c|c|c|c}
         \bI_K & \bI_K & \rule{1cm}{0.1pt} & \bI_K
    \end{array}
    \right].
  \end{align*}
  We have that $\bD_1^\T \bD_1 = (n/K)\bI_K$ and $\sigma(\bD_1) = \{\sqrt{n/K},0\}$, where $\sigma(\bD_1)$ denotes the set of singular values of $\bD_1$.
  Set $$\gamma_1 = \frac{1}{4\,K}$$
  and define the  $n \times K$ matrix $\bD_2$ by the relation
  \begin{align*}
    \bD_2^\T = \gamma_1 \left[
    \begin{array}{c|c}
         \mathbf{0}_{K,K} & \mathbf{1}_{K,(n-K)}
    \end{array}
    \right]
  \end{align*}
  where we denote by $\mathbf{1}_{n,p}$ (respectively, $\mathbf{0}_{n,p}$) the $n \times p$ matrix with all entries $1$ (respectively, $0$). Then, $\bD_2^\T \bD_2 = (n-K)\gamma_1^2 \mathbf{1}_{K,K}$ and $\sigma(\bD_2) = \{\gamma_1 \sqrt{K(n-K)},0\}$. 
  We will further consider the matrix $\bD_3 = \bD_1 + \bD_2$ given by the relation
  \begin{align*}
    \bD_3^\T &= \left[ 
    \begin{array}{l|l|l|l}
        1~ \cdots ~0 & (1+\gamma_1) \phantom{1+1}\gamma_1 ~\phantom{1+1}\dots ~
        \phantom{1++}\gamma_1\phantom{1+} & \dots \dots & (1+\gamma_1) \phantom{1+1}\gamma_1 ~\phantom{1+1}\dots ~
        \phantom{1++}\gamma_1\phantom{1+} \\
        \vdots \phantom{\cdots \cdots} \vdots & 
        \phantom{1+}\gamma_1 ~\phantom{1+1}(1+\gamma_1) ~\phantom{+}\dots\phantom{1++} ~\gamma_1\phantom{1} & \dots \dots & \phantom{1+}\gamma_1 ~\phantom{1+1}(1+\gamma_1) ~\phantom{+}\dots\phantom{1++} ~\gamma_1\phantom{1} \\
        \vdots \phantom{\cdots \cdots} \vdots & \phantom{1+}\vdots ~\phantom{1+1}\phantom{(1+\gamma_1)} ~\phantom{+\dots\phantom{1++}} ~~\vdots\phantom{1} & \dots \dots & \phantom{1+}\vdots ~\phantom{1+1}\phantom{(1+\gamma_1)} ~\phantom{+\dots\phantom{1++}} ~~\vdots\phantom{1} \\
        0~ \cdots ~1 & \phantom{1+}\gamma_1 \phantom{1+1+1}\gamma_1\phantom{1+1} \dots \phantom{1+}(1+\gamma_1) & \dots \dots & \phantom{1+}\gamma_1 \phantom{1+1+1}\gamma_1\phantom{1+1} \dots \phantom{1+}(1+\gamma_1)
    \end{array}
    \right] \\
    &= \left[ 
    \begin{array}{l|l|l|l}
         \bI_K & \bI_K + \gamma_1 \mathbf{1}_{K,K} & \dots & \bI_K + \gamma_1 \mathbf{1}_{K,K}.
    \end{array}
    \right].
  \end{align*}
  Applying Weyl's inequality~\cite[Theorem C.6]{Giraud-book}, we get
  \begin{align*}
    \left\{\begin{array}{l}
         \lambda_1(\bD_3) \leq \sqrt{n/K} + \frac{1}{4}\sqrt{(n-K)/K} \leq \frac{5}{4}\sqrt{n/K},   \\
         \lambda_K(\bD_3) \geq \sqrt{n/K} - \frac{1}{4}\sqrt{(n-K)/K} \geq \frac{3}{4}\sqrt{n/K}.
    \end{array}
    \right.
  \end{align*}
  Finally, the basic matrix $\bW^{(0)}$ is defined by the relation
  \begin{align*}
    ({\bW^{(0)}})^\T &= \bD_3^\T - \left[\begin{array}{l|l|l|l}
         \mathbf{0}_{K,K} &{K}\gamma_1 \bI_K & \dots & {K}\gamma_1 \bI_K
    \end{array}\right] \\
    &= \left[ 
    \begin{array}{l|l|l|l}
         \bI_K & (1-{K}\gamma_1)\bI_K + \gamma_1 \mathbf{1}_{K,K} & \dots & (1-{K}\gamma_1)\bI_K + \gamma_1 \mathbf{1}_{K,K}
    \end{array}
    \right].
  \end{align*}
  Clearly, $\bW^{(0)}$ satisfies Assumption \ref{ass:anchorDoc}, all entries of $\bW^{(0)}$ are non-negative and its rows sum up to $1$.
  Applying Weyl's inequality to matrix $\bW^{(0)}$ yields
  \begin{align}\label{eq:eqqq}
    \left\{\begin{array}{l}
         \lambda_1(\bW^{(0)}) \leq \frac{5}{4}\sqrt{n/K} +{ K\gamma_1(\sqrt{n/K-1})} \leq 
    \frac{3}{2}
    \sqrt{n/K},   \\
         \lambda_K(\bW^{(0)}) \geq \frac{3}{4}\sqrt{n/K} -{ K\gamma_1(\sqrt{n/K-1})} \geq 
         \frac{1}{2}
         \sqrt{n/K},
    \end{array}
    \right.
  \end{align}
  implying that $\kappa(\bW^{(0)})\leq 3$. 
  
  Our next step is to define the matrices $\bW^{(j)}, \ j=1,\dots, T$. Consider the set of binary sequences
  \begin{align*}
    M = \{0,1 \}^{K(n-K)/2}.
  \end{align*}
  Applying the Varshamov-Gilbert bound~\cite[Lemma 2.9]{tsybakov2008introduction} we get that there exist $w^{(j)} \in M$, $j = 1, \dots, T$, such that:
  \begin{align}
  \label{eq:eq1}
    \|w^{(i)}-w^{(j)}\|_1 = \|w^{(i)}-w^{(j)}\|_2^2 \geq \frac{K(n-K)}{16}, \text{ for any } 0 \leq i \neq j \leq T,
  \end{align}
  with $w^{(0)}=0$ and
  \begin{align}
  \label{eq:eqlogT}
    \log T \geq \frac{\log 2}{16} K(n-K).
  \end{align}
  We divide each $w^{(j)}$ into $(n-K)$ chunks as $w^{(j)} = (w_1^{(j)}, w_2^{(j)},\dots,w_{n-K}^{(j)})$ with $w_i^{(j)} \in \{0,1\}^{K/2}$. Next, for each $w_i^{(j)}$, we introduce its augmented counterpart defined as $\tilde{w}_i^{(j)} = (w_i^{(j)},-w_i^{(j)}) \in \{-1,0,1\}^{K}$. 
  In what follows, we
   set
  \begin{align*}
    \gamma = c_* \sqrt{\frac{N}{K(N-K)^2}},
  \end{align*}
  where $c_*>0$ is a small enough absolute constant.
 For $1\leq j\leq T$, define the $(n - K) \times K$ matrix $\Omega^{(j)}$ and the $n \times K$ matrix $\Delta^{(j)}$ as follows: 
  \begin{align*}
      \Omega^{(j)}= \gamma \begin{bmatrix}\tilde{w}_1^{(j)} \\ \tilde{w}_2^{(j)} \\ \vdots \\ \tilde{w}_{n-K}^{(j)}\\ \end{bmatrix} \text{ and } \Delta^{(j)} = \left[\begin{array}{l}
          \mathbf{0}_{K,K} \\
          \hline
          \Omega^{(j)}
      \end{array}\right].
  \end{align*}
  Note that all the entries of $(\Delta^{(j)})^\T \Delta^{(j)}$ are bounded in absolute value by $\gamma^2(n-K)$, which yields
  {
  \begin{align*}
     \| \Delta^{(j)} \|&=
     \sqrt{\|(\Delta^{(j)})^\T \Delta^{(j)}\|}\leq \sqrt{\|(\Delta^{(j)})^\T \Delta^{(j)}\|_F}
    \leq \gamma \sqrt{(n-K)K}.
  \end{align*}
  }
  Thus, choosing $c^*$ small enough and using the assumption that $N\geq 2K$ we obtain
  \begin{align*}
    \| \Delta^{(j)} \|  \leq \frac{1}{4}\sqrt{n/K}.
  \end{align*}
  Now, for $1\leq j\leq T$, we define $\bW^{(j)}$ as
  \begin{align}
  \label{eq:eqdefW}
    \bW^{(j)} = \bW^{(0)} + \Delta^{(j)}.
  \end{align}
  It is easy to check that, for each $1\leq j \leq T$, the rows of $\bW^{(j)}$ are probability vectors if $c^*$ is chosen small enough,
  and $\bW^{(j)}$ satisfies Assumption~\ref{ass:anchorDoc}.
  Moreover, using~\eqref{eq:eqqq} and applying Weyl's inequality once again, we obtain
  \begin{align}
  \label{eq:eq10}
    \left\{\begin{array}{l}
         \lambda_1(\bW^{(j)}) \leq 
    \frac{7}{4}\sqrt{n/K},   \\
         \lambda_K(\bW^{(j)}) \geq \frac{1}{4}\sqrt{n/K},
    \end{array}
    \right.
  \end{align}
   so that $\kappa(\bW^{(j)})\leq 7$.
  
  \smallskip
    
  2. {\it Constructing matrix $\bA$ and checking the fact that $\bPi^{(j)}\in \mathcal{M}$, $j=0,1,\dots, T$.} \\  
  Assume that $p$ is a multiple of $K$ (if it is not the case the definition of $\bA$ should be modified by adding a block of zeros of the size of the residual).
  Define the following block matrix:
  \begin{align*}
    \bA^{0}=\{\underbrace{\be_{1}, 0_K, \ldots, 0_K}_{ p/K }, \underbrace{\be_{2},0_K, \ldots, 0_K}_{p/K }, \ldots, \underbrace{\be_{K},0_K, \ldots, 0_K}_{ p/K }\} \in \mathbb{R}^{K\times p},
  \end{align*}
  where $(\be_1,\dots,\be_K)$ is the canonical basis of $\mathbb{R}^K$ and $0_K\in \mathbb{R}^K$ is the vector with all entries 0. 
  Define
  \begin{align*}
    \bA: = \frac{N-K}{N}\bA^0 + \frac{K}{pN}\mathbf{1}_{K,p}.
  \end{align*}
  All entries of $\bA$ are non-negative and the rows of $\bA$ sum up to $1$.  We have that $\sigma\left(\frac{N-K}{N} \bA^0\right) =\left\{\frac{N-K}{N},0\right\}$ and $\sigma\left(\frac{K}{pN} \mathbf{1}_{K,p}\right)=\left\{\frac{K^{3/2}}{\sqrt{p}N},0\right\}$.  
   Using the assumption that $K\leq p/4$ and
  Weyl's inequality we get
  \begin{align}
  \label{eq:eq11}
    \left\{\begin{array}{l}
      \lambda_1(\bA) \leq \frac{N-K}{N} + \frac{K^{3/2}}{\sqrt{p}N} \leq 1,  \\
         \lambda_K(\bA) \geq \frac{N-K}{N} - \frac{K^{3/2}}{\sqrt{p}N} \geq 1/4,
    \end{array}
    \right.
  \end{align}
  which implies that $\kappa(\bA)\leq 4$. 
  
  For $0\leq j \leq T$, define $\bPi^{(j)} = \bW^{(j)}\bA$.
  Using Lemma
  \ref{lem:eigenvalue_minoration}, \eqref{eq:eqqq}, \eqref{eq:eq10} and \eqref{eq:eq11} we obtain
  \begin{align}\label{lower_bound_minSV}
   \lambda_K(\bPi^{(j)}) = \lambda_K(\bW^{(j)}\bA) & \geq \lambda_K(\bW^{(j)})\lambda_K(\bA)\geq \frac1{16}\sqrt{n/K}.
  \end{align}
   It follows from \eqref{eq:eqqq}, \eqref{eq:eq10} and \eqref{lower_bound_minSV} that the first inequality in Assumption \ref{assumption_eigenvalues} is satisfied for $\bW =\bW^{(j)}$ and  $\bPi= \bPi^{(j)} = \bW^{(j)}\bA$, $j=0,1,\dots, T$. Next, using the first inequality in \eqref{lower_bound_minSV} and the fact that
   $\lambda_1(\bW^{(j)}\bA) \leq \lambda_1(\bW^{(j)}) \lambda_1(\bA)$ yields
  \begin{align}
  \label{eq:ineqkappa}
    \kappa(\bW^{(j)} \bA) \leq \kappa(\bW^{(j)})\kappa(\bA)\leq C.
  \end{align}
  Thus,  Assumption \ref{assumption_eigenvalues} is satisfied for $\bW =\bW^{(j)}$ and  $\bPi= \bPi^{(j)} = \bW^{(j)}\bA$, $j=0,1,\dots, T$. In conclusion, we have proved that $\bPi^{(j)}\in \mathcal{M}$, $j=0,1,\dots, T$.

\smallskip

  To prove Theorem \ref{theorem:lower_bound}, we now use Theorem 2.5 in~\cite{tsybakov2008introduction}, according to which the lower bounds \eqref{inf_bnd_l2} and \eqref{inf_bnd_l1} hold if the following conditions are satisfied:
  \begin{itemize}
    \item[(a)] $\text{KL}(\mathbb{P}_{\bPi^{(j)}},\mathbb{P}_{\bPi^{(0)}}) \leq \frac{\log T}{16}$, for each $j = 1, \dots, T$, where $\text{KL}(\mathbb{P},\mathbb{Q})$ denotes the Kullback-Leibler divergence between the probability measures $\mathbb{P}$ and $\mathbb{Q}$.
      
    \item[(b)] For $0\leq j < \ell \leq T$ we have $\min_{\mathbf{P}\in \mathcal{P}} \| \bW^{(\ell)}- \bW^{(j)}\mathbf{P}\|_F \geq c \sqrt{\frac{n}{N}}$ and
     $\min_{\mathbf{P}\in \mathcal{P}}\| \bW^{(j)} - \bW^{(\ell)}\mathbf{P}\|_1 \geq  c\, n \sqrt{\frac{K}{N}}.$
    where  $\mathcal{P}$ is the set of all permutation matrices and $c$ is a positive constant.
    
    \item[(c)]  The maps $(\bM_1,\bM_2) \mapsto \min_{\mathbf{P} \in \mathcal{P}}\| \bM_1 - \bM_2 \mathbf{P}\|_{F}$ and $(\bM_1,\bM_2) \mapsto \min_{\mathbf{P} \in \mathcal{P}}\| \bM_1 - \bM_2 \mathbf{P}\|_{1}$ are semi-distances.
  \end{itemize}
  The rest of the proof is devoted to checking that these conditions (a) -- (c) are indeed satisfied.
  
  \medskip
   
   3. {\it Proof of (a).}\\
  Our aim now is to derive an upper bound on the Kullback-Leibler divergence between $\mathbb{P}_{\bPi^{(j)}}$ and $\mathbb{P}_{\bPi^{(0)}}$, where 
  $$ \bPi^{(j)} = \bW^{(j)}\bA\quad \text{and}\quad  \bPi^{(0)} = \bW^{(0)}\bA.$$
  To shorten the notation,  we set 
  $$\alpha:= \frac{N-K}{N} + \frac{K}{pN}, \qquad \beta := \frac{K}{pN}.$$
  For any $1\leq i \leq n$, $1\leq \ell \leq p$, we have
  $\Pi_{i\ell}^{(0)} = \sum_{k=1}^K W_{ik}^{(0)}A_{k\ell}$.
  If $i\ge K+1$, for the entries in the $i$th row of matrix  $\bPi^{(0)}$ the following holds.
  \begin{itemize}
      \item For the columns $\ell$ such that $(\ell-1)$ is a multiple of $p/K$:
        \begin{itemize}
            \item $\Pi_{i\ell}^{(0)}$ takes once the value $\alpha + (K-1)\gamma_1(\beta-\alpha)$,
            \item $\Pi_{i\ell}^{(0)}$ takes $K-1$ times the value $\beta + \gamma_1(\alpha-\beta)$.
        \end{itemize}
    \item For all other columns: $\Pi_{i\ell}^{(0)}\in \{ \alpha, \beta \}$.  
  \end{itemize}
  \newpar
  
On the other hand,  for any $1\leq j \leq T$, by the definition of $\bW^{(j)}$ in~\eqref{eq:eqdefW} we have 
  \begin{align*}
    \bPi^{(j)}&= \bPi^{(0)} + \Delta^{(j)} \\
    &= \bPi^{(0)}
    +
    \left[
      \begin{array}{c}
        \mathbf{0}_{K, p}  \\
        \hline
        \Omega^{(j)}\bA 
      \end{array}
    \right].
  \end{align*}
  Therefore, for any $1 \leq \ell \leq p$, if $i\leq K$, $\Pi_{i\ell}^{(j)}=\Pi_{i\ell}^{(0)}$, and if $i \ge K+1$, $\Pi_{i\ell}^{(j)}=\Pi_{i\ell}^{(0)} + \Delta_{i \ell}^{(j)}$, where 
   \begin{align*}
    \Delta_{i \ell}^{(j)} = \gamma\left[\sum_{k=1}^{K/2 } w_{i-K}^{(j)}(k) A_{k\ell}-\sum_{k=K/2 +1}^{K} w_{i-K}^{(j)}(k-K/2 ) A_{k\ell}\right]~.
  \end{align*}
  If $i \ge K+1$, for the entries in the $i$th row of matrix $\Delta^{(j)}$ the following holds.
  \begin{itemize}
      \item For the columns $\ell$ such that $(\ell-1)$ is a multiple of $p/K$:
  \begin{itemize}
      \item $\Delta_{i \ell}^{(j)}$ is $K/2$ times equal to $\gamma(\alpha-\beta)$,
      \item $\Delta_{i \ell}^{(j)}$ is $K/2$ times equal to $-\gamma(\alpha-\beta)$.
  \end{itemize}
  \item For all other $\ell$:  $\Delta_{i \ell}^{(j)}=0$.
  \end{itemize}
  
We are now ready to bound the Kullback-Leibler divergence between $\mathbb{P}_{\bPi^{(j)}}$ and $\mathbb{P}_{\bPi^{(0)}}$. Denote by $M_p(N,q)$ the multinomial distribution with parameters $(N,q)$ where $q$ is a probability vector in $\mathbb{R}^p$.  We recall that the Kullback-Leibler divergence between two multinomial distributions  $M_p(N,q_1)$ and $M_p(N,q_2)$ is equal to $N \sum_{\ell=1}^p q_{1\ell} \log \left( q_{1\ell}/q_{2\ell} \right)$. Hence, we have
 \begin{align*}
     \text{KL}\left(\mathbb{P}_{\bPi^{(j)}},\mathbb{P}_{\bPi^{(0)}} \right) &= N \sum_{i=1}^n\sum_{\ell=1}^p \Pi_{i \ell}^{(j)} \log \left(\frac{ \Pi_{i \ell}^{(j)} }{ \Pi_{i \ell}^{(0)} }\right) \\
     &= N \sum_{i=K+1}^n\sum_{\ell=1}^p \Pi_{i \ell}^{(j)} \log \left(\frac{ \Pi_{i \ell}^{(j)} }{ \Pi_{i \ell}^{(0)} }\right) \\
     & = N \sum_{i=K+1}^n\sum_{\ell=1}^p\left(\Pi_{i \ell}^{(0)} + \Delta_{i \ell}^{(j)}  \right) \log \left(1+\frac{ \Delta_{i \ell}^{(j)} }{\Pi_{i \ell}^{(0)}}\right) \\
     & \leq N \sum_{i=K+1}^n \sum_{\ell=1}^p \left( \Delta_{i \ell}^{(j)}  + \frac{(\Delta_{i \ell}^{(j)})^2}{\Pi_{i \ell}^{(0)}} \right).
 \end{align*}
 Note that, by construction,  $\sum_{\ell=1}^p \Delta_{i \ell}^{(j)} = 0$. Therefore,
 \begin{align*}
     \text{KL}\left(\mathbb{P}_{\bPi^{(j)}},\mathbb{P}_{\bPi^{(0)}} \right) & \leq N\sum_{i=K+1}^n \sum_{\ell=1}^p \frac{(\Delta_{i \ell}^{(j)})^2}{\Pi_{i \ell}^{(0)}} \\
     & = N\sum_{i=K+1}^n \  \sum_{(\ell-1) \text{ multiple of } p/K} \ \frac{(\Delta_{i \ell}^{(j)})^2}{\Pi_{i \ell}^{(0)}} \\
     & \leq N \sum_{i=K+1}^n \  \sum_{(\ell-1) \text{ multiple of } p/K} \ \frac{\gamma^2 (\alpha-\beta)^2}{\Pi_{i \ell}^{(0)}} \\
     & \leq \frac{4c^*}{K} \sum_{i=K+1}^n \  \sum_{(\ell-1) \text{ multiple of } p/K}  \frac{(N-K)^2}{N^2\,\Pi_{i \ell}^{(0)}} \\
     & \leq \frac{4c^*}{K} \sum_{i=K+1}^n \left[\frac{3N}{N-K} + \frac{2(K-1)KN}{N-K} \right] \\
     & \leq  \frac{c}{K} (n-K)\frac{K^2N}{N-K} \\
     & \leq c K(n-K) \\
     & \leq \frac{\log T}{16},
 \end{align*}
where we have used \eqref{eq:eqlogT} and we have chosen   $c^*$ small enough, such that the constant $c$ in the penultimate line does not exceed $(\log 2)/256$.

\smallskip

4. {\it Proof of (b).}  \\
Note that for any $0\leq j < \ell \leq T$ we have $\min_{\mathbf{P}\in \mathcal{P}}\|\bW^{(\ell)}-\bW^{(j)}\mathbf{P}\|_F = \|\bW^{(\ell)}-\bW^{(j)}\|_F$ since the first $K$ rows are the same for matrices $W^{(\ell)}$ and $W^{(j)}$. Then,
  \begin{align*}
    \| \bW^{(j)} - \bW^{(\ell)}\|_F^2 &= \| \Omega^{(j)} - \Omega^{(\ell)}\|_F^2
    = \sum_{i=1}^{n-K} \|\Omega_{i \cdot}^{(j)} - \Omega_{i \cdot}^{(\ell)} \|_2^2 \\
    & = 2\gamma^2 \sum_{i = 1}^{n-K} \| w_i^{(j)} - w_i^{(\ell)}\|_2^2
    = 2 \gamma^2 \| w^{(j)}-w^{(\ell)}\|_2^2 \\
    & \geq \frac{\gamma^2}{8} K (n - K) \quad \text{(using~\eqref{eq:eq1})}\\ 
    &= \frac{c_*^2}{8} \frac{N (n - K)}{(N - K)^2}\geq c\frac{n}{N} \quad \text{(since $K\le n/2$)}, \numberthis \label{eq:eq6}
  \end{align*}
  which proves (b) for the Frobenius norm. Quite analogously, for the $\ell_1$-norm we get 
  \begin{align*}
    \| \bW^{(j)} - \bW^{(\ell)}\|_1 &= \| \Omega^{(j)} - \Omega^{(\ell)}\|_1
    = \sum_{i=1}^{n-K} \|\Omega_{i \cdot}^{(j)} - \Omega_{i \cdot}^{(\ell)} \|_1 \\
   & \geq  2\gamma \|w^{(j)}-w^{(\ell)}\|_1 \\
    & \geq  c\, n \sqrt{\frac{K}{N}}.
  \end{align*}

\smallskip

5. {\it Proof of (c).}  \\
  We now prove that the map $(\bM_1,\bM_2)\mapsto \min_{\mathbf{P}\in \mathcal{P}}\| \bM_1 - \bM_2 \mathbf{P}\|_F$ satisfies the triangle inequality. For any matrices $\bM_1,\bM_2,\bM_3$, we have
  \begin{align*}
      \min_{\mathbf{P}\in \mathcal{P}}\| \bM_1 - \bM_2\mathbf{P} \|_F &= \min_{\mathbf{P,P'}\in \mathcal{P}}\| \bM_1\mathbf{P'} - \bM_2 \mathbf{P}\|_F \\
      &\leq \min_{\mathbf{P,P'}\in \mathcal{P}}\left( \|\bM_1\mathbf{P'} -\bM_3 \|_F+\|\bM_3-\bM_2 \mathbf{P} \|_F\right) \\
      &= \min_{\mathbf{P'}\in \mathcal{P}} \| \bM_1\mathbf{P'} - \bM_3 \|_F + \min_{\mathbf{P}\in \mathcal{P}} \|\bM_3 - \bM_2\mathbf{P} \|_F \\
      &= \min_{\mathbf{P'}\in \mathcal{P}} \| \bM_1 - \bM_3\mathbf{P'} \|_F + \min_{\mathbf{P}\in \mathcal{P}} \|\bM_3 - \bM_2\mathbf{P} \|_F.
  \end{align*}
The same calculation holds with the $\ell_1$-norm in place of the Frobenius norm. This completes the proof of Theorem~\ref{theorem:lower_bound}.
  

\section{Auxiliary lemmas}
\begin{lemma}
\label{lemma:matrix_eigen0}
  Let Assumption~\ref{ass:anchorDoc} be satisfied. Then,
  \begin{align}
  \label{eq:matrix_property3a}
       \lambda_{K}(\bW) &\ge 1. 
  \end{align}
  If, in addition, $\lambda_K(\bPi)>0$ then the matrix $\bf U$ of left singular vectors of $\bPi$ can be represented in the form~\eqref{UWH}, where {\bf H} is a rank $K$ matrix with singular values 
  \begin{align}\label{eq:matrix_property6}
    \lambda_{1}({\bf H}) = \frac{1}{\lambda_{K}(\bW)}, \qquad \lambda_{\min}({\bf H}) = \lambda_{K}({\bf H})= \frac{1}{\lambda_{1}(\bW)},
  \end{align} 
  and the condition number satisfying
  \begin{align}
    \label{eq:matrix_property4}
    \kappa({\bf H}) &= \kappa(\bW).
  \end{align}
\end{lemma}
\begin{proof}
  Let $J^*
  \subseteq \{1,\dots,n\}$ be the set of $K$ row indices of $\bW$ corresponding to anchor documents. By Assumption~\ref{ass:anchorDoc} we have $\bW_{J^*}=\bI_K$. Hence, 
  \begin{align*}
    \lambda_{K}(\bW) &=\min_{\|a\|_2=1}  \|\bW a\|_2 \ge \min_{\|a\|_2=1} \|\bW_{J^*} a\|_2 =1,
  \end{align*}
  which proves~\eqref{eq:matrix_property3a}. Next, if $\lambda_K(\bPi)>0$ then matrix ${\bf L}$ is positive definite and we define ${\bf H}:= \bA {\bf V}{\bf L}^{-1}$. In view of~\eqref{svc:Pi} we have $\bW {\bf H} = \bPi {\bf V}{\bf L}^{-1} = {\bf U}$, which yields~\eqref{UWH}. We now prove that ${\bf H}$ is non-degenerate. Indeed, \eqref{eq:matrix_property3a} implies that matrix $\bW^\T \bW\in \mathbb{R}^{K\times K}$ is positive definite, so that ${\bf H} =(\bW^\T \bW)^{-1}\bW^\T {\bf U}$. Then for the minimal singular value $\lambda_{\min}({\bf H})$ of matrix ${\bf H}$ we have
  \begin{align*}
    \lambda_{\min}({\bf H}) &=\min_{\|a\|_2=1}  \|(\bW^\T \bW)^{-1}\bW^\T {\bf U} a\|_2 
    \\
    &\ge  \min_{x\in \mathbb{R}^n: \|x\|_2=1} \|(\bW^\T \bW)^{-1}\bW^\T x\|_2 = \frac{1}{\lambda_{1}(\bW)}>0.
  \end{align*}
  Thus, ${\bf H}$ is non-degenerate and we can write $\bW = {\bf U}{\bf H}^{-1}$ implying~\eqref{eq:matrix_property6}. 
  Equality~\eqref{eq:matrix_property4} is an immediate consequence of~\eqref{eq:matrix_property6}.
\end{proof}
\begin{lemma}
\label{lemma:matrix_eigen}
  Let $\bW, \bA$ and $\bPi=\bW\bA$ be matrices with non-negative entries satisfying~\eqref{factoriz1}. 
  Then the singular values of matrices \(\documentTopicMatrix\) and \(\bPi\) satisfy the  inequalities 
    \begin{align}
      \lambda_K(\bPi) &\leq \sqrt{n/K}, \label{eq:matrix_property1}\\
      \sqrt{n/K} \le &\ \lambda_{1}(\bW) \leq \sqrt{n}, \label{eq:matrix_property2}\\
           \lambda_{1}(\bPi) &\leq \sqrt{K}\lambda_{1}(\bW). \label{eq:matrix_property3}
    \end{align}
  \end{lemma}
  \begin{proof}
    Inequality~\eqref{eq:matrix_property1} follows from the fact that
    \begin{align*}
      K\lambda_K^2(\bPi) \leq \lambda_1^2(\bPi)+\dots+\lambda_K^2(\bPi) = \|\bPi\|_F^2 \leq n.
    \end{align*}
    Next, using~\eqref{factoriz1} we obtain 
    \begin{EQA}[c]\label{eq:matrix_property5}
      \lambda_{1}(\bW) \leq \|\bW\|_F
      =
      \sqrt{\sum_{i = 1}^{\nsize} \sum_{k = 1}^{\ntopics} W_{ik}^2}
      \leq
      \sqrt{\sum_{i = 1}^{\nsize} \sum_{k = 1}^{\ntopics} W_{ik}}
      =
      \sqrt{\nsize}.
    \end{EQA}
    On the other hand, for $a=(1/\sqrt{K},\dots,1/\sqrt{K})^{\rm T}\in \mathbb{R}^K$ we have
    \begin{align*}
      \lambda_{1}(\bW) \ge   \|\bW a\|_2 = \sqrt{\sum_{i = 1}^{\nsize} \frac1K \left(\sum_{k = 1}^{\ntopics} W_{ik}\right)^2}=\sqrt{\frac{n}{K}}.
    \end{align*}
   Quite similarly to~\eqref{eq:matrix_property5}, using~\eqref{factoriz1} we get $\|\bA\|=\lambda_{1}(\bA)\leq \|\bA\|_F \le \sqrt{K}$, which implies~\eqref{eq:matrix_property3}:
    \begin{EQA}
      \|\bPi\|&=&\|\bW\bA\|\leq \|\bW\|\|\bA\|\leq \sqrt{K}\lambda_{1}(\bW).
    \end{EQA}
%
  \end{proof}
  
  {
  \begin{lemma}
  \label{lem:eigenvalue_minoration}
  Let $K\le \min(n,p)$. For any two matrices $\bW\in \mathbb{R}^{n\times K}$ and $\bA\in \mathbb{R}^{K\times p}$ we have %
  \begin{align}\label{lower_bound_minSVD}
    \lambda_K(\bW\bA) & \geq \lambda_K(\bW)\lambda_K(\bA).
  \end{align}
  \end{lemma}
  \begin{proof} We consider only the case $\lambda_K(\bA)>0$ since otherwise \eqref{lower_bound_minSVD} is trivial.
  By Courant-Fischer min-max formula (see, e.g.,  \cite[Theorem C.3]{Giraud-book}) we have
  \begin{align*}
    \lambda_K(\bW\bA) &= \max_{S: {\rm dim}(S)=K}\,\min_{y \in S\setminus \{0\} } \frac{\|\bW\bA y\|_2}{\|y \|_2}, 
  \end{align*}
  where the maximum is taken over all linear spans $S$ of $K$ vectors in $\mathbb{R}^p$. Since $\lambda_K(\bA)>0$ and $\bA y \in \mathbb{R}^K$ we can write
  \begin{align*}
    \lambda_K(\bW\bA) &= \max_{S: {\rm dim}(S)=K}\,\min_{y \in S\setminus \{0\} } \frac{\|\bW\bA y\|_2}{\|\bA y \|_2} \frac{\|\bA y\|_2}{\|y \|_2}\\
    & \ge  \min_{x \in \mathbb{R}^K\setminus \{0\} } \frac{\|\bW x\|_2}{\|x \|_2} \max_{S: {\rm dim}(S)=K}\,\min_{y \in S\setminus \{0\} }  \frac{\|\bA y\|_2}{\|y \|_2} \\
    & = \lambda_K(\bW)\lambda_K(\bA).
  \end{align*}
  \end{proof}

 \begin{lemma}
\label{lem:anchor_word}
Let
$\bA^{0}$ be a matrix with the following block structure:
  \begin{align*}
    \bA^{0}=\big[\alpha_1\be_{1},  \ldots, \alpha_K \be_{K},\underbrace{0_{K},\ldots,0_{K}}_{p-K}\big] \in \mathbb{R}^{K\times p},
  \end{align*}
  where $(\be_1,\dots,\be_K)$ is the canonical basis of $\mathbb{R}^K$, $\alpha_i\in (0,1)$ and $0_K\in \mathbb{R}^K$ is the vector with all entries 0. Let $$\bA=\bP_1(\bA^{0}+\bA^1)\bP_2\quad \text{and}\quad \bPi=\bW\bA$$ where $\bP_1, \bP_2$ are permutation matrices, $\Vert \bA^1\Vert \leq \beta$, and $\bW\in \mathbb{R}^{n\times K}$. If $\underset{1\le i\le K}{\min}\;\alpha_i- \beta\geq C$ then $\lambda_{\ntopics}(\bPi) \ge  C\lambda_{\ntopics}(\bW)$.
\end{lemma}
\begin{proof}
Matrix $\bA^0$ has $K$ top non-zero singular values $\alpha_1,\ldots,\alpha_K$. 
Using Weyl's inequality (see, e.g.,  \cite[Theorem C.6]{Giraud-book}) we get
\[\lambda_K(\bA)=\lambda_K(\bA^{0}+\bA^1)\geq \underset{1\le i\le K}{\min}\;\alpha_i-\beta\geq C.\]
Combining this inequality with \eqref{lower_bound_minSVD} yields the result.
\end{proof}}

\bibliography{topicModels}
\bibliographystyle{plain}


\begin{appendices}

\section{Additional Experiments: Estimation of  topic-word matrix}


\label{sec:topic_word_experiments}
  In this section, we investigate the SPOC estimator of topic-word matrix $\bA$ using the sequence of experiments on synthetic data {similar to those} of Section~\ref{sec:simulations}. Figures~\ref{fig:A_fig_n_1}-\ref{fig:A_fig_K_1} below present the results of simulations with different values of parameters $n,p,N$ and the number of topics $K$. {The generation of matrices $\bW$ and $\bA$ was performed in the same way as in Section~\ref{sec:simulations}.} For each value on the $x$-axes of the figures, we present the averaged result over 10 simulations.
  
  Our objective was to assess the effect of each of parameters $n,p,N,K$ on the Frobenius error between $\bA$ and the estimator $\hat{\bA}$ derived from SPOC algorithm via~\eqref{A_estimate}. For comparison, we provide the same simulation study for the LDA estimator of $\bA$. The experiments show that the SPOC estimator is very competitive with LDA while being more stable.
  One may also notice that the Frobenius error of both SPOC and LDA estimators has a behavior similar to the optimal rates for  estimation of matrix $\bA$ under $\ell_1$-error derived in~\cite{Ke2017,bing2018fast}. Indeed, the estimation error is decreasing as function of $n$ and $N$ and it is increasing as function of $p$ and $K$.

  \begin{figure}[t]
    \centering
    \includegraphics[scale=0.3]{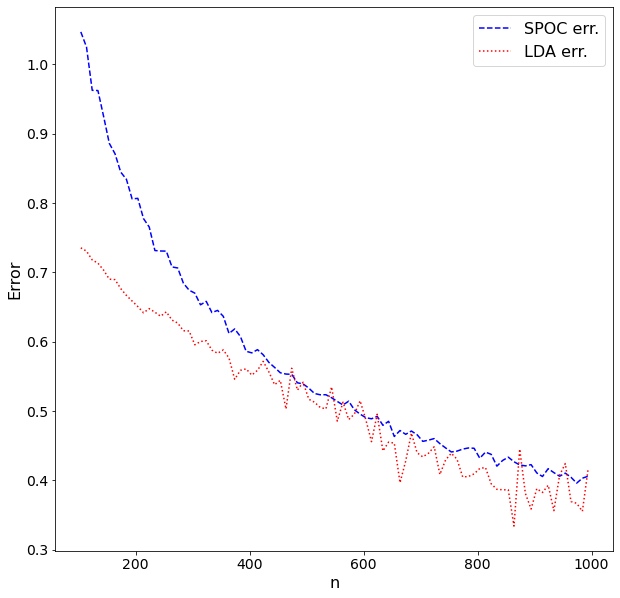}
    \caption{The $n$-dependency of the Frobenius error of $\hat {\bA}$ 
    using SPOC and LDA algorithms. Total number of words $p=5000$, number of sampled words in each document $N=200$. 
    }
  \label{fig:A_fig_n_1}
  \end{figure}
  
  \begin{figure}[H]
    \centering
    \includegraphics[scale=0.3]{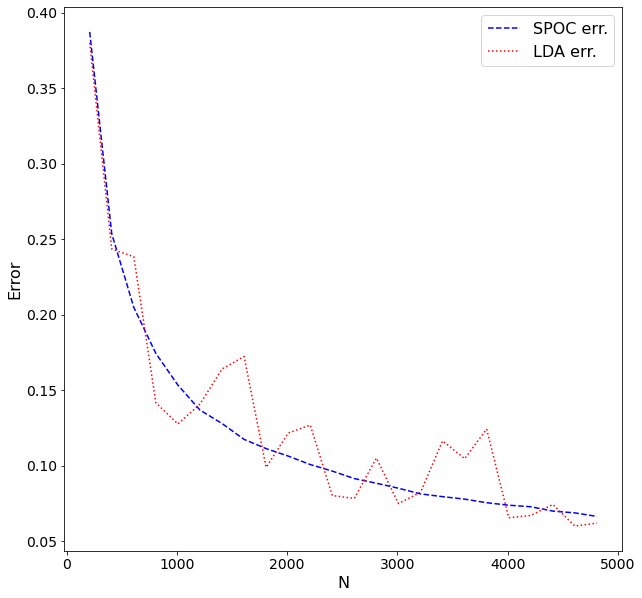}
    \caption{The $N$-dependency of the Frobenius error of $\hat {\bA}$ 
    using SPOC and LDA algorithms. Total number of words $p=5000$, number of documents $n=1000$.
    }
  \label{fig:A_fig_N_1}
  \end{figure}

  \begin{figure}[H]
    \centering
    \includegraphics[scale=0.3]{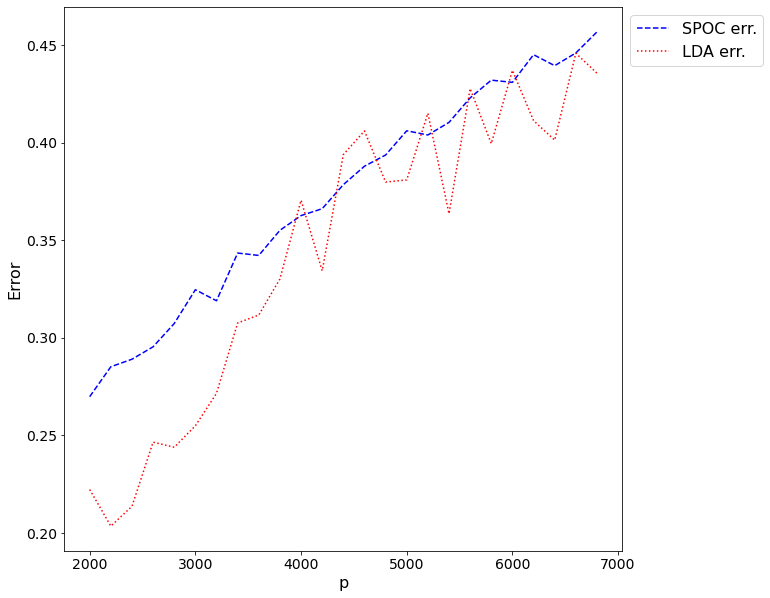}
    \caption{The $p$-dependency of the Frobenius error of $\hat {\bA}$ 
    using SPOC and LDA algorithms. Total number of words $n=1000$, number of sampled words in each document $N=200$.
    }
  \label{fig:A_fig_p_1}
  \end{figure}

  \begin{figure}[H]
    \centering
    \includegraphics[scale=0.3]{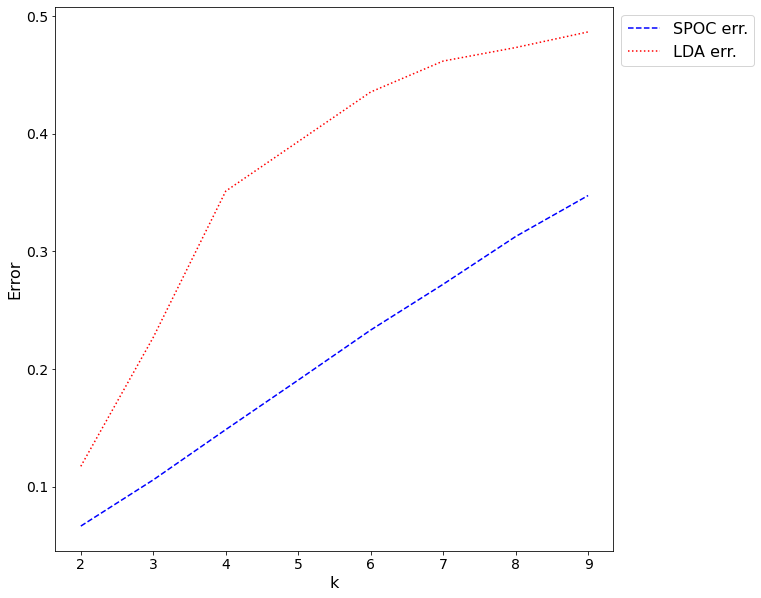}
    \caption{The $K$-dependency of the Frobenius error of $\hat {\bA}$ using SPOC and LDA algorithms. Total number of words $p=5000$, number of sampled words in each document $N=5000$, number of documents $n=1000$.
    }
  \label{fig:A_fig_K_1}
  \end{figure}

\section{Additional Experiments: Empirical study of singular values of word-document and topic-document matrices }\label{simulations_sv}
\label{sec:assumption_eigenvalues_check}
  Most conditions and assumptions used throughout the paper are satisfied for fairly general choices of parameters. However, Assumption~\ref{assumption_eigenvalues} enforces certain bounds on matrices $\bW$ and $\bPi$ which might seem restrictive. The goal of this section is to experimentally show that singular values and quotients appearing in Assumption~\ref{assumption_eigenvalues} admit reasonably small upper bounds.

  We consider  matrices \(\bW, \bPi\) and \(\bA\) generated in the following way. In most experiments we take $K=3$ and the matrix $\bW$ has the following structure: $K$ rows of $\bW$ are canonical basis vectors, each of the remaining $n-K$ rows is generated independently using the Dirichlet distribution with parameter $\alpha = (0.1,0.15,0.2)$. In the experiments where $K$ must vary, we define $\bW$ in a different way. Namely, for the $n-K$ rows that are not canonical basis vectors, each element \(W_{kj}\) is generated from the uniform distribution on $[0, 1]$ and then each row of the matrix is normalized to have \(\sum_{k = 1}^K W_{ik} = 1\). For the matrix $\bA$, we take $K$ columns proportional to canonical basis vectors with coefficients equal to random variables $U_k, k = 1, \dots, K$ uniformly distributed on $[0, 1]$. The elements \(A_{kj}\) of matrix \(\bA\) in the remaining $p-K$ columns are  obtained by generating numbers from the uniform distribution on $[0, 1]$ and then normalizing each row of the matrix to have \(\sum_{j = {K + 1}}^p A_{kj} = 1 - U_k, k=1,\dots,K\). The resulting matrix \(\bA\) has  normalized rows such that \(\sum_{j = 1}^p A_{kj} = 1\).

  We essentially use  the same parameters as in the experiments reported in Section~\ref{sec:simulations}. The dependencies of the condition numbers $\kappa(\bPi)$ and $\kappa(\bW)$ on parameters $n$, $p$ and $K$ are presented on Figures~\ref{fig:Pi_cond} and~\ref{fig:W_cond}. All the condition numbers have small to moderate values for a quite wide range of parameters $n$ and $p$, while the dependence on $K$ is stronger. Additionally, we study the ratio $\lambda_{1}(\bW) / \lambda_{\ntopics}(\bPi)$ also appearing in Assumption~\ref{assumption_eigenvalues}. As presented on Figure~\ref{fig:Pi_W_quotient} it shows the tendencies similar to the condition numbers.

  \begin{figure}[t]
    \centering
    \includegraphics[scale=0.5]{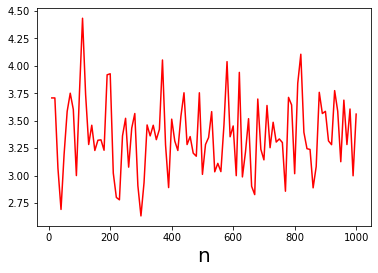}
    \includegraphics[scale=0.5]{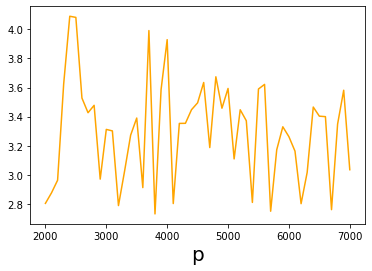}
    \includegraphics[scale=0.3]{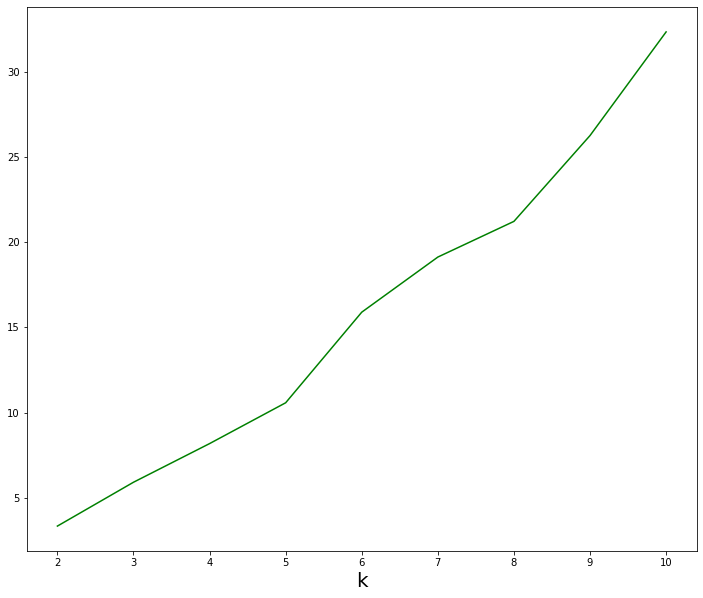}
    \caption{The dependency of $\kappa(\bPi)$ on parameters $n$, $p$ and $K$.}
  \label{fig:Pi_cond}
  \end{figure}

  \begin{figure}[t]
    \centering
    \includegraphics[scale=0.5]{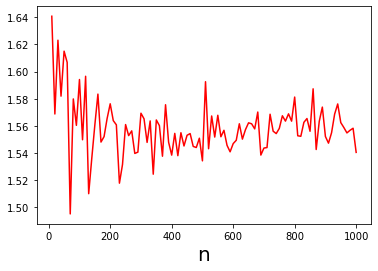}
    \hspace{30pt}
    \includegraphics[scale=0.3]{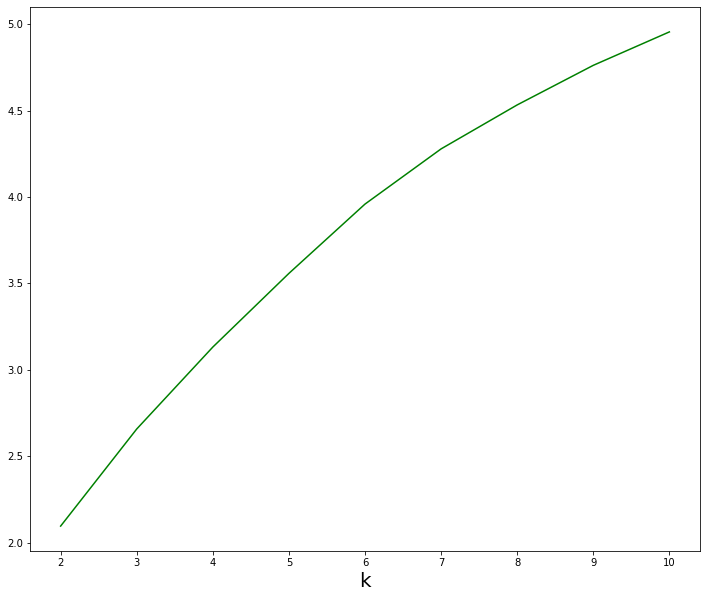}
    \caption{The dependency of $\kappa(\bW)$ on parameters $n$ and $K$.}
  \label{fig:W_cond}
  \end{figure}

  \begin{figure}[t]
    \centering
    \includegraphics[scale=0.30]{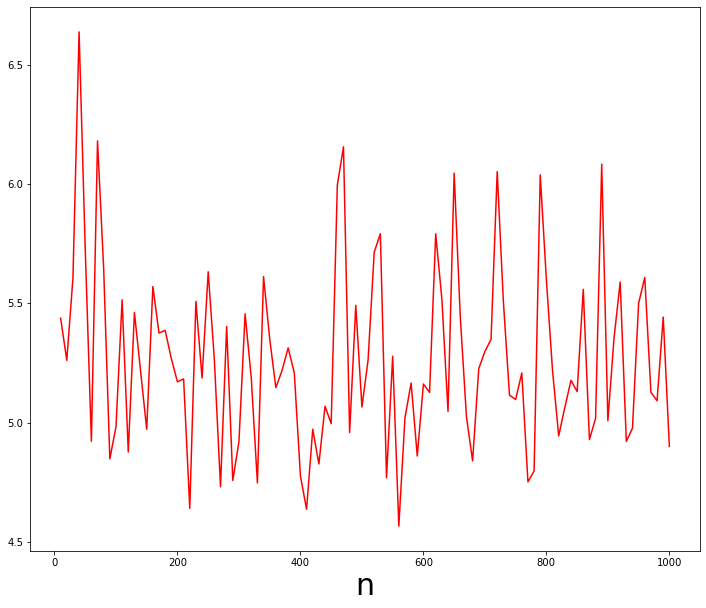}
    \includegraphics[scale=0.30]{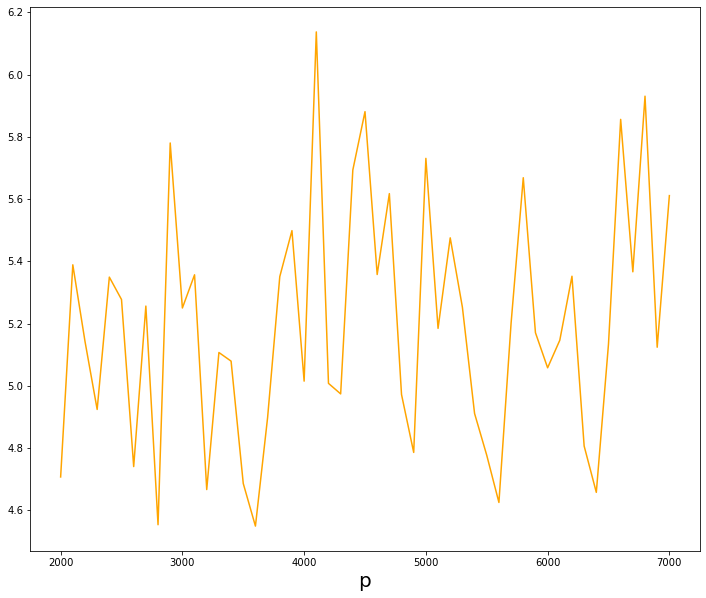}
    \includegraphics[scale=0.30]{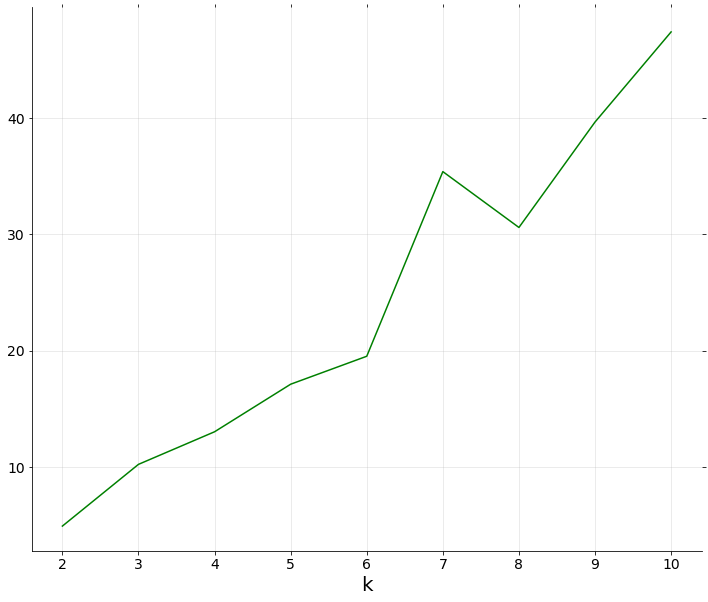}
    \caption{The dependency of $\lambda_{1}(\bW) / \lambda_{\ntopics}(\bPi)$ on parameters $n$, $p$ and $K$.}
  \label{fig:Pi_W_quotient}
  \end{figure}
\end{appendices}
\end{document}